\documentclass[MSNbibl,number,citesort,seceqn,dvips]{arxbj}
\usepackage{stmaryrd}

\aid{0}
\volume{19}
\issue{5B}
\pubyear{2013}
\firstpage{2414}
\lastpage{2436}
\doi{10.3150/12-BEJ457} 

\makeatletter

\newcommand{\rrvert}{\vert}
\newcommand{\llvert}{\vert}
\newcommand{\M}{\mathcal{M^{\hspace{-2.5mm}\circ}}\hspace{1mm}}
\newtheorem{theore}{Theorem}[section]
\newproclaim{defi}[theore]{Definition}
\newtheorem{propo}[theore]{Proposition}
\newremark{ejem}[theore]{Example}
\newremark{nota}[theore]{Remark}
\newtheorem{lema}[theore]{Lemma}
\makeatother

\begin{document}
\begin{frontmatter}

\title{Stochastic integration with respect to additive functionals of
zero quadratic variation}
\runtitle{Stochastic integration}

\begin{aug}
\author{\fnms{Alexander} \snm{Walsh}\corref{}\ead[label=e1]{awalshz@tx.technion.ac.il}}
\runauthor{A. Walsh} 
\address{Department of Industrial Engineering and Management,
Technion Israel Institute of Technology, Haifa, Israel. \printead{e1}}
\end{aug}

\received{\smonth{5} \syear{2012}}

%
\begin{abstract}
We consider a Markov process $X$ associated to a nonnecessarily
symmetric Dirichlet form $\mathcal{E}$. We define a stochastic
integral with
respect to a class of additive
functionals of zero quadratic variation and then we obtain an It\^{o}
formula for the process $u(X)$, when $u$ is locally in the domain of
$\mathcal{E}$.
\end{abstract}

%
\begin{keyword}
\kwd{additive functional}
\kwd{Dirichlet form}
\kwd{Fukushima decomposition}
\kwd{It\^o formula}
\kwd{Markov process}
\kwd{stochastic calculus}
\kwd{quadratic variation}
\kwd{zero energy process}
\end{keyword}

\end{frontmatter}
%

\section{Introduction and main results}

The semimartingale theory has produced a fundamental tool based on
stochastic integration and It\^{o}'s formula: the stochastic calculus.
Since Markov processes are not in general semimartingales, Fukushima
\cite{F} developed another stochastic calculus in the framework of
symmetric Dirichlet spaces. Let $E$ be a Polish space, $m$ be a Radon
measure on $E$ and $(X,\Omega,\mathcal{F}_t,\mathbf{P}_x,t\in
\mathbb{R}_+,x\in E)$ be a
$E$-valued $m$-symmetric Markov process with regular Dirichlet form
$\mathcal{E}
$. For any element $u$ of the domain $\mathcal{F}$ of $\mathcal{E}$,
the process
$(u(X_t)-u(X_0),t\geq0)$ admits the decomposition
\[
u(X_t)-u(X_0)=M^u_t+N_t^u,
\]
where $M^u$ is a martingale additive functional of finite energy and
$N^u$ is a continuous additive functional of zero energy. This
decomposition is called Fukushima's decomposition and it can be seen as
a substitute of the Doob--Meyer decomposition of super-martingales and
It\^{o}'s formula for semimartingales.
The part of the class of bounded variation processes in the
semimartingale theory is played by the class of continuous additive
functionals of zero energy. In general, these additive functionals are
not of bounded variation and therefore the Lebesgue--Stieltjes integrals
can not be defined. Nevertheless, the concepts of energy and quadratic
variation are closely related, see Graversen and Rao \cite{GR}. In
particular, it is well known that for any function $g\in L^1(E;m)$, the
process $N^u$ has 0-quadratic variation with respect to the measure
$\mathbf{P}_{g\cdot m}:=\int_E \mathbf{P}_x(\cdot)m(\mathrm{d}x)$, that is,
$u(X)$ is a
Dirichlet process in the F\"{o}llmer sense \cite{Fo}. Is shown in
\cite{Fo} (see also \cite{Fo2}) that for any function $\varphi$ in
$\mathcal{C}^1(\mathbb{R})$, the following limit exists $\mathbf
{P}_x$-a.s.  for $m$-a.e. $x
\in E$:
%
\begin{equation}
\label{Fint} \int_0^t \varphi
\bigl(u(X_s)\bigr)\,\mathrm{d}N^u_s:=
\lim_{n\rightarrow\infty
}\sum_{i=0}^{n-1} \varphi
\bigl(u(X_{ti/n})\bigr) \bigl(N^u_{t(i+1)/n}-N^u_{ti/n}
\bigr).
\end{equation}
In his It\^o formula expending $u(X)$ \cite{Fo,Fo2}, this
integral replaces the Lebesgue--Stieltjes integral in the classical It\^
{o} formula for semimartingales. In this connection, Russo and Vallois
\cite{RV} and \cite{RV2} have obtained a similar stochastic calculus
through a so-called regularization procedure.

However, in the theory of symmetric Markov processes has been necessary
to extend (\ref{Fint}) to more general integrands than whose of type
$\varphi\circ u$, for $\varphi\in C^1(\mathbb{R})$. For example, in
order to
define stochastic line integrals of 1-forms along the paths of
symmetric diffusion processes on manifolds, Nakao \cite{N} introduced
an integral $\int_0^t f(X_s)\,\mathrm{d}N^u_s$ for $f$ bounded function
element of~$\mathcal{F}$. His new integral is an additive functional
in the
sense of Fukushima \textit{et al.} \cite{FOT}, and therefore it is defined
$\mathbf{P}_x$-a.s.  for q.e. (quasi every) $x \in E$ (i.e., outside
of an
exceptional set \cite{FOT}). This allows to get a refinement of the F\"
{o}llmer's It\^{o} formula, for which, the expansion of $u(X)$ holds
$\mathbf{P}_x$-a.s.  for q.e. $x \in E$. Besides, this integral is used by
Fitzsimmons and Kuwae \cite{FiK}, to study the lower order
perturbation of diffusion processes.\looseness=-1

We emphasize that following \cite{Fo}, for $f$ and $u$ in $\mathcal
{E}$ it is
possible to show the existence of the limit $\lim_{n\rightarrow\infty
}\sum_{i=0}^{n-1}
[f(X_{ti/n})(N^u_{t(i+1)/n}-N^u_{ti/n})+u(X_{ti/n})(N^f_{t(i+1)/n}-N^f_{ti/n})].$
$\mathbf{P}_x$-a.e. for $m$-a.e. $x\in E$, however it is not possible
to show
the existence of the limits
$\lim_{n\rightarrow\infty}\sum_{i=0}^{n-1}
f(X_{ti/n})(N^u_{t(i+1)/n}-N^u_{ti/n})$ and $\lim_{n\rightarrow\infty
}\sum_{i=0}^{n-1} u(X_{ti/n})(N^f_{t(i+1)/n}-N^f_{ti/n})$ and
therefore to define the integrals $\int_0^t f(X_s)\,\mathrm{d}N^u_s$
and $\int_0^t u(X_s)\,\mathrm{d}N^f_s$ in separated way.

The conditions of existence of Nakao's integral being too restrictive,
this notion could not be used by Chen \textit{et al.} \cite{CFKZ2} to study the
lower order perturbation of symmetric Markov processes that are not
diffusions. Chen \textit{et al.} \cite{CFKZ}, have extended Nakao's integral to
a larger class of integrators as well as integrands. Using time
reversal, they have defined an integral $\int_0^t f(X_s)\,\mathrm
{d}C_s$ for
$f$ in $\mathcal{F}_{\mathit{loc}}$, the set of functions locally in $\mathcal
{F}$ and $C$ in a
large class of processes containing $\{N^u\dvt u\in\mathcal{F}\}$. The process
$C$ is not in general of zero energy but of zero quadratic variation
and the integral is not an additive functional or a local additive
functional but a local additive functional admitting null set. Kuwae
\cite{Ku3} gives a refinement of Chen \textit{et al.} work, redefining the
stochastic integral without using time reversal.

Our aim in this paper, is to construct an integral $\int_0^t
f(X_s)\,\mathrm{d}
C_s$ for
a Markov process $X$ associated to a nonnecessarily symmetric regular
Dirichlet form
$(\mathcal{E},\mathcal{F})$ in a Hilbert space $L^2(E,m)$, $f$
locally in $\mathcal{F}$ and $C$
local continuous additive functional with zero quadratic variation. To
do so, one can not extend the construction of Chen \textit{et al.} neither
Kuwae's construction because they both heavily rely on the symmetry of
the Markov process.

On one hand, it is legitimate to solve this question since many results
for symmetric Dirichlet forms hold for nonsymmetric Dirichlet forms,
see, for example, \cite{K,Ku2,Ku,MaR} and~\cite
{O}. In particular, Fukushima's decomposition holds for nonsymmetric
regular Dirichlet forms, but also the correspondence between Markov
processes and (nonnecessarily symmetric) Dirichlet forms, Revuz
correspondence and other relations between probabilistic notions for
$X$ and analytic notions for $\mathcal{E}$.

In order to introduce our stochastic integral, we need the following
definitions: A~sequence $(\Pi_n:=\{
0=t_{n,0}<t_{n,1}<\cdots<t_{p_n,n}<\infty\})_{n\in\mathbb{N}}$ of
partitions of
$\mathbb{R}_+$ is said to tend to the identity if $\|\Pi_n\|:=\max\{
t_{n,k+1}-t_{n,k}\}\rightarrow0\mbox{ as }n\rightarrow\infty\mbox{
and }t_{p_n,n}\rightarrow\infty$. We denote by $\mathcal{N}$ the set of
continuous additive functionals of finite energy and of zero quadratic
variation. Denote by $\mathcal{N}_{f\mbox{-}\mathit{loc}}$ the set of process locally
in $\mathcal{N}$. (See Definitions \ref{flocal} and \ref{defafzqv}
below.)

\begin{theore}\label{integralestocastica}
For a function $f$ locally in $\mathcal{F}$ and an element $C$ of
$\mathcal{N}_{f\mbox
{-}\mathit{loc}}$, there exists an unique local additive functional $I$ such that:

For any sequence $(\Pi_n)$ of partitions tending to the identity,
there exists a subsequence $(\Pi_{n_k})$\vspace*{-2pt} such that $\mathbf
{P}_x$-a.e. for
$m$-a.e. $x$ in E:
$\sum_{i=0}^{p_{n_k}-1}f(X_{t_{n_k,i}})[C(t\wedge
t_{n_k,i+1})-C(t\wedge t_{n_k,1})]$
converges to\vspace*{1pt} $I_t$ uniformly on any compact of $[0,\zeta)$.
Moreover, $I$ belongs to $\mathcal{N}_{f\mbox{-}\mathit{loc}}$.
\end{theore}

The local additive functional obtained in the theorem below is denoted
by $f*C_t$ or by $\int_0^t f(X_s)\,\mathrm{d}C_s$. Then when $C$ is of bounded
variation, $f*C$ coincides with the Lebesgue--Stieltjes integral. If
there exists a local martingale $M$, a real function $u$ on $E$ and a
$C^1$-real function $\varphi$ such that $f=\varphi\circ u$ and
$u(X)=M+C$, $\mathbf{P}_x$-a.e. for $m$-a.e. then $f*C$ coincides with
the F\"
{o}llmer integral.

This new integral leads to an It\^o formula for $u(X)$ when $u$ belongs
to $\mathcal{F}_{\mathit{loc}}$. On this purpose, we need first an extension of the
Fukushima decomposition of $u(X)$ for the elements $u$ locally in
$\mathcal{F}
$. This extension is well-known for diffusions processes.

When $X$ is not a diffusion we have the following substitute of the
Fukushima decomposition: Denote by $\M_{f\mbox{-}\mathit{loc}}$ the set of
local martingale additive functionals locally of finite energy.

\begin{propo}\label{itoLevy}
For $u$ in $\mathcal{F}_{\mathit{loc}}$, the process $u(X)$ admits the following
decomposition $\mathbf{P}_x$-a.e. for q.e. $x\in E$:
\[
u(X_t)=u(X_0)+V_t^u+W^{u}_t+C^u_t,
\qquad  t<\zeta \ (t<\infty\mbox{ if }u\in\mathcal{F}),
\]
where $W^u\in\M_{f\mbox{-}\mathit{loc}}$, $C^u\in\mathcal{N}_{c,f\mbox
{-}\mathit{loc}}$ and
$V^u$ is the AF of bounded variation given by:
\[
V_t^u=\sum_{s\leq t}
\bigl(u(X_s)-u(X_{s-})\bigr)1_{\{|u(X_s)-u(X_{s-})|>1\}
}-u(X_{\zeta-})1_{\{t\geq\zeta\}}.
\]
Moreover, the jumps of $W^u$ are bounded by $1$.
\end{propo}

In particular, if $E=\mathbb{R}^d$ and we take $u$ the coordinate function
$\pi_i\dvtx x\rightarrow x_i$, $i=1,\ldots ,d$, the above result can be seen as
a generalization of the It\^{o}--L\'{e}vy decomposition for L\'{e}vy
processes (e.g., Sato \cite{Sa}).

Using the notation of Proposition \ref{itoLevy}, we introduce the
following extension of the It\^{o} formula.

\begin{propo}\label{ito}
Suppose that $\Phi$ belongs to $\mathcal{C}^2(\mathbb{R}^d)$ and
$u=(u_1,\ldots ,u_d)$
belongs to $\mathcal{F}^d_{\mathit{loc}}$. Then for q.e. $x\in E$, $\mathbf
{P}_x$-a.s.  for all
$t\in[0,\zeta)$ ($[0,\infty)$ if $u\in\mathcal{F}^d$):
%
\begin{eqnarray}\label {fito}
&&\Phi\bigl(u(X_t)\bigr)-\Phi\bigl(u(X_0)\bigr)\nonumber
\\
&&\quad =\sum
_{k=1}^d\int_0^t
\frac{\partial
\Phi}{\partial x_k}\bigl(u(X_{s-})\bigr)\,\mathrm{d}W_s^{u_k}+
\sum_{k=1}^d\int_0^t
\frac{\partial\Phi}{\partial x_k}\bigl(u(X_{s})\bigr)\,\mathrm{d} C_s^{u_k}
\nonumber
\\
&&\qquad {}+\frac{1}{2}\sum_{k,\ell=1}^d\int
_0^t\frac{\partial^2 \Phi
}{\partial x_k\,\partial x_{\ell}}\bigl(u(X_{s})
\bigr)\,\mathrm{d}\bigl\langle W^{u_k,c},W^{u_{\ell},c}\bigr
\rangle_s
\\
&&\qquad {}+\sum_{s\leq t} \Biggl[ \Phi\bigl(u(X_s)
\bigr)-\Phi\bigl(u(X_{s-})\bigr)\nonumber
\\
&&\qquad \hspace*{31pt}{}-\sum_{k=1}^d
\frac{\partial\Phi}{\partial
x_k}\bigl(u(X_{s-})\bigr)\Delta u_k(X_s)1_{\{|\Delta(u_k(X_s))|<1\}}
\Biggr]
\nonumber
\\
&&\qquad {}-\sum_{k=1}^d \frac{\partial\Phi}{\partial x_k}
\bigl(u(X_{\zeta
-})\bigr)u(X_{\zeta-})1_{\{t\geq\zeta\}}.
\nonumber
\end{eqnarray}
\end{propo}

In the case that $E=\mathbb{R}^d$, if we take $u=(\pi_1,\ldots,\pi_d)$, we
obtain a It\^{o} formula for the process $X$ and therefore the
Fukushima decomposition of $\Phi(X)$ for $\Phi\in C_0^2(\mathbb{R}^d)$.
Following Albeverio \textit{et al.} \cite{AFRS}, the Dirichlet form $\mathcal
{E}(\Phi
,\Psi)$ for $\Phi$ and $\Psi$ in $\mathcal{F}$ can be approximate
by $\frac
{1}{t}\int_E \mathbf{E}_x(\Phi(X_0)-\Phi(X_t))\Psi(x)m(\mathrm
{d}x)$, then we
hope that the It\^{o} formula can be used, for example, in order to
give a probabilistic approach to the work of Hu \textit{et al.} \cite{HMS,HMS2} concerning Beurling--Deny decomposition for nonsymmetric
Dirichlet forms.

As mentioned above, following \cite{Fo}, it is possible to define the
second term on the right-hand side of (\ref{fito}) as the limit
$\mathbf{P}_x$-a.s.  for $m$-a.e. $x\in E$:
%
\begin{equation}\label{Fintm}
\lim_{n\rightarrow\infty} \sum_{i=0}^{n-1}\sum
_{k=1}^d\int_0^t
\frac{\partial\Phi}{\partial x_k}\bigl(u(X_{ti/n})\bigr) \bigl[C_{t(i+1)/n}^{u_k}-C_{ti/n}^{u_k}
\bigr]
\end{equation}
and to get (\ref{fito}), $\mathbf{P}_x$-a.s.  for $m$-a.e. $x\in E$. If we
want to use the results of \cite{Fo} to ensure the existence of the
above limit and (\ref{fito}), $\mathbf{P}_x$-a.s.  for q.e. $x\in E$,
we need
to show that the process $N^u$ is of zero quadratic variation, and then
$u(X)$ is a Dirichlet process in the F\"{o}llmer sense, $\mathbf{P}_x$-a.s.
for q.e. $x\in E$. This has been demonstrated by many authors for
certain class of processes and functions $u$. For example, this has
been shown by Bouleau and Yor \cite{BY} for the case when $X$ is a
unidimensional semimartingale with discontinuous part of bounded
variation and $u$ an absolutely continuous function with bounded weak
derivative. Also it has been shown for F\"{o}llmer and Protter \cite
{FP} for a multidimensional Brownian Motion and a function $u$ locally
in the Sobolev space $\mathcal{W}^{1,2}$. There are even, similar
results for a time dependent function, maybe the first one shown by F\"
{o}llmer \textit{et al.} \cite{FPS} for the case of a unidimensional Brownian
Motion and a time dependent function $u$ with locally square integrable
weak derivatives satisfying a mild condition of continuity. The results
of \cite{FPS} have been extended, for example, by Bardina and Jolis
\cite{BJ}, Bardina and Rovira \cite{BR} for elliptic diffusion
processes, by Ghomrasni and Peskir \cite{GP} for continuous
semimartingales, by Eisenbaum \cite{E,E2}, Eisenbaum and
Walsh \cite{EW} for L\'{e}vy processes with Brownian component, by
Eisenbaum and Kyprianou \cite{EK} and Walsh \cite{W1} for L\'{e}vy
processes without Brownian component. The precedent list is not
exhaustive, in fact in the references of these cited papers, we can
find more examples of stochastic processes $X$ and functions $u$,
time-dependent or not, for which $u(X\cdot,\cdot)$ is a Dirichlet process in
the F\"{o}llmer sense and therefore, Proposition~\ref{ito} is already known.

The integrals used in (\ref{fito}) are based on the Fukushima
stochastic calculus using the concept of additive functional, the
additive property is then essential. In this context, is not possible
to extend Proposition~\ref{ito} to time-dependent functions $u$.

The paper is organized as follows. In Section \ref
{preliminares}, we present some preliminaries. In Section~\ref
{sectionsi}, we construct a stochastic integration with respect to
$N^u$. To do so, we first establish a decomposition of $N^u$ as the sum
of three processes $N^u_1$, $N^u_2$ and $N^u_3$ such that $N^u_1$ and
$N^u_2$ are respectively associated to the diffusion part and the
jumping part of the symmetric part of $\mathcal{E}$, and $N^u_3$ is of bounded
variations. Next, we present respectively stochastic integration with
respect to $N^u_1$ and~$N^u_2$. These results lead to an integral with
respect to $N^u$ which is used with an argument of localization to
introduce the stochastic integral with respect to $C$ in $\mathcal{N}_{c,f\mbox{-}\mathit{loc}}$. In Section \ref{prueba1}, we prove Theorem \ref
{integralestocastica}, that is, the stochastic integral with respect to
$C$ can be approximated by Riemman sums. We also show that when the
Dirichlet form is symmetric, the obtained stochastic integral with
respect to $C$ coincides with the integral defined by Chen \textit{et al.} \cite
{CFKZ}. In Section \ref{prueba2}, we establish Proposition \ref
{itoLevy} and the It\^{o} formula in which this new integral takes the
place of the Lebesgue--Stieltjes integral in the classical It\^{o}
formula for semimartingales.

\section{Preliminaries}\label{preliminares}

In this paper, we use mostly notation and vocabulary from the book of
Fukushima \textit{et al.} \cite{FOT} still available in the nonnecessarily
symmetric case (see Ma and Rockner \cite{MaR} and Oshima \cite{O}).
This section contains existing results or some immediate consequences
of existing results that will be useful for the other sections.

Throughout this paper, we assume that $X=(\Omega,\{\mathcal{F}_t\}_{t\geq0},\{
X_t\}_{t\geq0}, \{\mathbf{P}_z\}_{z\in E})$ is a Hunt process on a locally
compact separable metric space $E$, properly associated to a regular
Dirichlet form $\mathcal{E}$ with domain $\mathcal{F}$ in a Hilbert
space $L^2(E;m)$.
We do not assume that $\mathcal{E}$ is symmetric. Set $\mathcal
{E}_1(u,v):=\mathcal{E}
(u,v)+(u,v)$, where $(\cdot,\cdot)$ denotes the inner product in $L^2(E,m)$. It
is known that $\mathcal{F}$ is a Hilbert space with inner product
$\tilde{\mathcal{E}
}_1(u,v):=\frac{1}{2}(\mathcal{E}_1(u,v)+\mathcal{E}_1(v,u))$.
Denote by $\zeta$ the
life time of $X$ and $\partial$ the extra point such that $X_t(\omega
)=\partial$ for all $t\geq\zeta(\omega)$ and $\omega\in\Omega$.
A real function on $E$ is extended to a function on $E\cup\partial$
by setting $f(\partial)=0$.

The energy of an AF (additive functional) $A$ is defined by
\[
e(A):=\lim_{t\rightarrow0}\frac{1}{2t}\mathbf{E}_m
\bigl[A_t^2 \bigr]
\]
when the limit exists and for two AF $A,B$, their mutual energy is
defined by
\[
e(A,B):=\tfrac{1}{2}\bigl[e(A+B)-e(A)-e(B)\bigr].
\]
An AF $M$ is called a martingale additive functional (abbreviated as
MAF) if it is finite, c\`{a}dl\`{a}g and for q.e. $x$ in $E$: $\mathbf{E}_x[M_t^2]<\infty$ and $\mathbf{E}_x[M_t]=0$ for any $t\geq0$. Denote
by $\M
$ the set of MAF's of finite energy and
\[
\mathcal{N}_c:= \bigl\{N\dvt
 N\mbox{ is a finite continuous AF}, e(A)=0,
\mathbf{E}_x(|N_t|)<\infty\mbox{ q.e. for
each }t>0\bigr\}.
\]

For any $u\in\mathcal{F}$, $M^u$ and $N^u$ denote the elements of $\M
$ and
$\mathcal{N}_c$, respectively, that are present in Fukushima
decomposition of
$u(X_t)-u(X_0), t\geq0$, that is:
\[
u(X_t)-u(X_0)=M_t^u+N_t^u
\qquad \mbox{for }t\geq0, \mathbf {P}_x\mbox{-a.e. for q.e. }x\in E.
\]

In this paper, we always assume that the elements of $\mathcal{F}$ are always
represented by its quasi-continuous $m$-versions.

For a nearly Borel set $B(\subset E)$, $\sigma_B$ and $\tau_B$
represent the first hitting time to $B$ and the first exit time from
$B$ respectively, that is:
\begin{eqnarray*}
\sigma_B&:=&\inf\{t>0\dvt X_{t}\in B\},
\nonumber
\\
\tau_B&:=&\inf\{t>0\dvt X_{t}\notin B\}.
\nonumber
\end{eqnarray*}

We denote by $\mathcal{F}_b$ the subset of $\mathcal{F}$ of bounded
functions and for a
nearly Borel finely open set, $\mathcal{F}_G$ the set of functions
$u\in\mathcal{F}$
such that $u(x)=0$ for q.e. $x\in E\setminus G$. The subset of
$\mathcal{F}_G$
of bounded functions is denoted by $\mathcal{F}_{b,G}$. The
abbreviations CAF
and PCAF stand for continuous additive functionals and positive
continuous additive functional, respectively. The Revuz measure of a
PCAF is the measure given by the Revuz correspondence between PCAFs and
smooth measures. All these definitions are found in \cite{FOT}.

The following theorem is a small modification of Theorem 5.4.2 of \cite
{FOT} established for the symmetric case, but it holds also for the
nonsymmetric case. (See \cite{O} and \cite{Wt}.)

\begin{theore}\label{T5.4.2FOT}
Let $u$ be an element of $\mathcal{F}$ and let $G$ be a nearly Borel finely
open set. Let $A^1$ and $A^2$ be two PCAFs with Revuz measure $\mu_1$
and $\mu_2$, respectively, such that $\mathcal{F}_{b,G}\subset
L^1(E,\mu_i)$
for $i=1,2$. Then $\mathbf{P}_x(N_t^u=A^1_t-A^2_t\mbox{ for }t<\tau_G)=1$
for q.e. $x\in E$, if and only if:
\[
\mathcal{E}(u,h)=\langle \mu_2-\mu_1,h\rangle \qquad  \forall h\in
\mathcal{F}_{b,G}.
\]
\end{theore}

\begin{defi}\label{NcF*}
We define $\mathcal{N}_{c}^0$ as the set of CAFs $C$ such that, there exists
$u$ in $\mathcal{F}$ and finite PCAFs $A^1, A^2$ with Revuz measure
$\mu_1$
and $\mu_2$, respectively, satisfying: $\mathcal{F}_{b}\subset
L^1(E,\mu_i)$ and
$\mathbf{P}_x(C_t=N_t^u+A^2_t-A^1_t\mbox{ for }t<\infty)=1\mbox{
for q.e. }x\in E$.
In this case, we define the linear functional $\Theta(C)$ on $\mathcal{F}_{b}$
by
\[
\big\langle \Theta(C),h\big\rangle :=-\mathcal{E}(u,h)+\langle \mu_2-\mu_1,h\rangle ,
\qquad h\in \mathcal{F}_{b}.
\]
\end{defi}

It follows from Theorem \ref{T5.4.2FOT} that the definition of $\Theta
(C)$ for $C\in\mathcal{N}_{c}^0$ is consistent in the sense that it
does not
depend of the elements which represent $C$.\vadjust{\goodbreak}

The following lemma is an immediate consequence of Theorem \ref{T5.4.2FOT}.

\begin{lema}\label{T2.2N}
Let $C^{(1)}$ and $C^{(2)}$ be elements of $\mathcal{N}_c^0$ and $G$ a nearly
Borel finely open set. Then $C^{(1)}=C^{(2)}$ on $\llbracket0,\sigma_{E
\setminus G}\llbracket$ $\mathbf{P}_x$-a.e. for q.e. $x\in E$ if and
only if
\[
\bigl\langle \Theta\bigl(C^{(1)}\bigr),h\bigr\rangle =\bigl\langle \Theta\bigl(C^{(2)}\bigr),h\bigr\rangle \qquad
\mbox{for all }h\in \mathcal{F}_{b,G}.
\]
\end{lema}

We recall that an increasing sequence of nearly Borel finely open sets
$(G_n)_{n\in\mathbb{N}}$ is called a nest if $\tau_{G_n}\uparrow
\zeta$
$\mathbf{P}_x$-a.s. for q.e. $x\in E$.

\begin{defi}\label{flocal}
Let $\Gamma$ be a class of local AF's. Following \cite{CFKZ}, we say
that a $(\mathcal{F}_t)$-adapted process $A$ is locally in $\Gamma$, and
write: $A\in\Gamma_{f\mbox{-}\mathit{loc}}$, if there exists a sequence $A^n$
in $\Gamma$ and a nest of nearly Borel finely open sets $\{G_n\}$ such
that $A_t=A_t^n$ for $t<\tau_{G_n}$ $\mathbf{P}_x$-a.e. for q.e.
$x\in E$.
In this case, $A$~is hence a local AF. (See \cite{CFKZ}, page 939, for
the definition of local~AF.)
\end{defi}

\begin{defi}\label{defafzqv}A local AF $V$ is said to be of zero
quadratic variation if for any $t>0$:
$\sum_{k=0}^{n-1}(V_{t(i+1)/n}-V_{ti/n})^2$ converges to zero as
$n\rightarrow\infty$ in $\mathbf{P}_{g\cdot m}$ measure on $\{t<\zeta\}$ for
some (and therefore for all) strictly positive $g\in L^1(E,m)$.
\end{defi}

We denote by $\mathcal{N}$ the set of CAFs of finite energy and of zero
quadratic variation. In \cite{Wt}, we have established the following
theorem of representation for the elements of $\mathcal{N}$.

\begin{theore}\label{teore1.1walsh}
Let $C$ be an element of $\mathcal{N}_{f\mbox{-}\mathit{loc}}$. There exists a
nest of
nearly Borel finely open sets $(G_n)$ and $(u_n)\in\mathcal{F}$ such that
$\mathbf{P}_x$-a.e. for q.e. $x\in E$:
\[
C_t=N_t^{u_n}-\int_0^tu_n(X_s)
\,\mathrm{d}s\qquad \mbox{for all }t<\tau_{G_n}.
\]
\end{theore}

\section{Stochastic integration}\label{sectionsi}

Consider an element $u$ of $\mathcal{F}$ and two finite smooth measure
$\mu_1$
and $\mu_2$ such that $\mathcal{E}(u,h)=\langle \mu_1-\mu_2,h\rangle $ for any
element $h$
of $ \mathcal{F}_{b}$. Thanks to Theorem \ref{T5.4.2FOT}, we know
that $N^u$
is of bounded variation. The integral $(f*N^u)_t:=\int_0^ t
f(X_s)\,\mathrm{d}
N_s^ u$ is hence well-defined as a Lebesgue--Stieltjes integral,
moreover, if $f$ belongs to $\mathcal{F}_b$, $f*N^ u$ belongs to
$\mathcal{N}_c^0$
(see Definition \ref{NcF*}) and for any $h$ in $\mathcal{F}_b$ we have:
%
\begin{equation}\label{012102}
\bigl\langle \Theta\bigl(f*N^u\bigr),h \bigr\rangle =\bigl\langle  \Theta\bigl(N^u\bigr),fh\bigr\rangle .
\end{equation}
Thanks to Lemma \ref{T2.2N}, the above equation characterizes the
local CAF $f*N^u$. In order to define the integral of $f$ with respect
to a process $N^u$ which is not necessarily of bounded variation, it is
hence natural to construct a local CAF still denoted by $f*N^u$
satisfying the equation (\ref{012102}). This has been done by Nakao
\cite{N} for the symmetric case and the aim of this section is to do
it for the nonnecessarily symmetric case.

The construction of $f*N^u$ is based on a decomposition of $N^u$ in
three components (see Lemma \ref{l030308} below). The first component
is associated to the diffusion part of $\tilde{\mathcal{E}}$, the symmetric
component of $\mathcal{E}$. The second one is associated to the jump
part of
$\tilde{\mathcal{E}}$ and the third one is a local CAF of bounded variation.
Once this decomposition done, the construction of $f*N^u$ will be close
to Nakao's construction in the symmetric case.

Thanks to a localization argument and Theorem \ref{teore1.1walsh}, we
will construct the integral $f*C$ for any $f\in\mathcal{F}_{\mathit{loc}}$ and
$A\in
\mathcal{N}_{f\mbox{-}\mathit{loc}}$.
We always consider $\mathcal{F}$ to be equipped with the norm $\tilde
{\mathcal{E}}_1$.
We will use repeatedly the following facts:
 \begin{enumerate}[(4)]
\item[(1)] If a PCAF $A$ with Revuz measure $\mu$ satisfies $\mu(E)<\infty$
then $A$ is finite continuous. Indeed, it is consequence of Lemma 4.3
of \cite{K}. This is the case when $A=\langle M \rangle$ for $M\in\M$.

\item[(2)] If $A$ is a PCAF $A$ with Revuz measure $\mu$ of finite energy
integral (i.e., there exists $U_1\mu$ in $\mathcal{F}$ such that
$\int_Eh(x)\mu(\mathrm{d}x)=\mathcal{E}_1(U_1\mu,h)$ for all $h\in
\mathcal{F}$), then $A$ is
finite continuous. In fact, for any $t$, $\mathbf{E}_x(A_t)\leq
\mathrm{e}^tU_1\mu
(x)<\infty$ q.e.

\item[(3)] For two CAF, $A,B$ and a nearly Borel set $G$, we have for q.e.
$x\in E$,\linebreak[4]  $\mathbf{P}_x(A = B\mbox{ on
}\llbracket
0,\tau_G\llbracket)=1$ if and only if for q.e. $x\in E$, $\mathbf
{P}_x(A=B\mbox{
on }\llbracket0,\sigma_{E\setminus G}\llbracket)=1$.

\item[(4)] If $J\dvtx \mathcal{F}\rightarrow\mathbb{R}$ is a continuous linear
functional, there
exists a unique $w\in\mathcal{F}$ such that $J(h)=\mathcal
{E}_1(w,h)$ for any $h\in\mathcal{F}
$. (See Theorem I.2.6. in \cite{MaR}.)
\end{enumerate}

\subsection{A decomposition of $N^u$}

We denote by $\tilde{\mathcal{E}}$ the symmetric part of $\mathcal
{E}$ and denote by
$\tilde{\mathcal{E}}^{(c)}$ and $\tilde{\mathcal{E}}^{(j)}$ the
diffusion part and
the jumping part of $\tilde{\mathcal{E}}$, respectively, in the Beurling--Deny
decomposition of $\tilde{\mathcal{E}}$. (See Section 5.3 in \cite{FOT}.)
For $u$ in $\mathcal{F}$, the applications $h\rightarrow\tilde
{\mathcal{E}
}^{(c)}(u,h)$ and $h\rightarrow\tilde{\mathcal{E}}^{(j)}(u,h)$ are
continuous. This leads to the following lemma.

\begin{lema}\label{lcont}
For $u$ in $\mathcal{F}$, there exists unique elements $w$ and $v$ of
$\mathcal{F}$
such that
$\mathcal{E}_1(w,h)=\tilde{\mathcal{E}}^{(c)}(u,h)$ and $\mathcal
{E}_1(v,h)=\tilde{\mathcal{E}
}^{(j)}(u,h)$ for any $h\in\mathcal{F}$.
\end{lema}

\begin{defi}\label{cNu}
For any $u\in\mathcal{F}$, set: $^c\tilde{N}^u_t:=N_t^w-\int_0^t
w(X_s)\,\mathrm{d}
s$ and $^j\tilde{N}^u_t:=N_t^v-\int_0^t v(X_s)\,\mathrm{d}s$ where $w$ and
$v$ are the elements of $\mathcal{F}$ given by Lemma \ref{lcont}.
\end{defi}

It is clear that $^c\tilde{N}^u$ and $^j\tilde{N}^u$ belongs to
$\mathcal{N}_c^0$ and
\begin{eqnarray}\label{lcNu}
\bigl\langle \Theta\bigl({}^c\tilde{N}^u\bigr),h\bigr\rangle &=&-\tilde{
\mathcal {E}}^{(c)}(u,h)\quad \mbox{and }
\nonumber\\[-8pt]\\[-8pt]
\bigl\langle \Theta\bigl({}^j\tilde{N}^u\bigr),h\bigr\rangle &=&-\tilde{
\mathcal {E}}^{(j)}(u,h)\qquad \mbox{for all }h\in\mathcal{F}_b.
\nonumber
\end{eqnarray}
For $u$ in $ \mathcal{F}$, the application $h\rightarrow\mathcal
{E}_1(u,h)$ is
continuous. Hence, there exists a unique $u^*$ in $\mathcal{F}$ such that
%
\begin{equation}
\label{eq2} \mathcal{E}_1(u,h)=\tilde{\mathcal{E}}_1
\bigl(u^*,h\bigr),\qquad  h\in\mathcal{F}.
\end{equation}
Moreover, we have:
%
\begin{equation}\label{020311}
\mathcal{E}_1\bigl(u^*,u^*\bigr)\leq K^2
\mathcal{E}_1(u,u),\vadjust{\goodbreak}
\end{equation}
where $K$ is a continuity constant of $\mathcal{E}$, which means that
$\mathcal{E}$
satisfies the sector condition:
\[
\bigl|\mathcal{E}_1(v,w)\bigr|\leq K \bigl(\mathcal{E}_1(v,v)
\bigr)^{1/2}\bigl(\mathcal {E}_1(w,w)\bigr)^{1/2}\qquad
\mbox{for all }v,w\in\mathcal{F}.
\]

\begin{lema}\label{l030308}
For $u$ in $\mathcal{F}$, let $u^*$ be given by (\ref{eq2}). Denote
by $\tilde
{k}$ the killing measure of $\tilde{\mathcal{E}}$ and by $\tilde{K}$
the PCAF
associated to $\tilde{k}(\mathrm{d}x)$ by the Revuz correspondence.
Then we
have $\mathbf{P}_x$-a.e. for q.e. $x\in E$ for any $t<\infty$
%
\begin{eqnarray}\label{032907}
N_t^u&=&{}^c\tilde{N}^{u^*}+
{}^j\tilde{N}^{u^*}-\int_0^tu^*(X_s)
\,\mathrm{d}\tilde{K}_s+\int_0^t
\bigl(u-u^*\bigr) (X_s)\,\mathrm {d}s.
\end{eqnarray}
\end{lema}

\begin{pf}
From the Beurling--Deny decomposition of $\tilde{\mathcal{E}}$, we
have that
for any $h\in\mathcal{F}$,
\[
\int_E \bigl|h(x)u^*(x)\bigr|\tilde{k}(\mathrm{d}x)\leq\bigl[
\mathcal {E}_1(h,h)\bigr]^{1/2}\bigl[\mathcal{E}_1
\bigl(u^*,u^*\bigr)\bigr]^{1/2}
\]
thus $\int_0^t|u^*(X_s)|\,\mathrm{d}\tilde{K}_s$ is a finite PCAF.
Then $\int_0^tu^*(X_s)\,\mathrm{d}\tilde{K}_s$ is an element of $\mathcal
{N}_c^0$ and then, the
right-hand side of (\ref{032907}) belongs to $\mathcal{N}_c^0$.
Denote this
element by $C$. The killing part $\tilde{\mathcal{E}}^{(k)}$ of
$\tilde{\mathcal{E}}$
satisfies
\[
\tilde{\mathcal{E}}^{(k)}\bigl(u^*,h\bigr)=\int_E
h(x)u^*(x)\tilde{k}(\mathrm {d}x)\qquad \mbox{for any }h\in\mathcal{F}.
\]
It follows from (\ref{lcNu}) that for all $h\in\mathcal{F}$:
\begin{eqnarray*}
\bigl\langle \Theta(C),h\bigr\rangle &=&-\tilde{\mathcal{E}}\bigl(u^*,h\bigr)+\bigl(u-u^*,h\bigr)
\nonumber
\\
&=&-\mathcal{E}(u,h).
\nonumber
\end{eqnarray*}
Then (\ref{032907}) follows from Lemma \ref{T2.2N}.
\end{pf}

\subsection{\texorpdfstring{Stochastic integration with respect to $^c\tilde{N}^u$}
{Stochastic integration with respect to cNu}}\label{subsectioncNu}

The following lemma is Lemma 5.1.2 and Corollary 5.2.1 of \cite{FOT}
that we recall for reader's convenience. In \cite{FOT}, it is
established for the symmetric case but is also valid for the
nonsymmetric case. In fact, its proof is based on the inequality
(5.1.1) of \cite{FOT} which is proved, for example, in Lemma 4.7 of
\cite{K} for the nonsymmetric case.

\begin{lema}\label{l5.1.2FOT}
Let $(u_n)$ be a sequence of quasi continuous functions in $\mathcal
{F}$ and
$\tilde{\mathcal{E}}_1$-convergent to $u$. Then there exists a
subsequence $\{
u_{n_k}\}$ such that for q.e. $x\in E$,
\[
\mathbf{P}_x\bigl(u_{n_k}(X_t)\mbox{ converges
uniformly to $u(X_t)$ on each compact interval of }[0,\infty)
\big)=1
\]
and the same holds for $N^{u_{n_k}}$ and $N^u$, and for $M^{u_{n_k}}$
and $M^u$, replacing $u_{n_k}(X)$ and $u(X)$, respectively.
\end{lema}

%

Using Lemma 4.3 of \cite{K}, we can obtain the following lemma.

\begin{lema}\label{convCAF}
Let $A^n$ be a sequence of PCAFs. Suppose that $\mu_{n}(E)$ converges
to zero as $n\rightarrow\infty$, where $\mu_n$ represents the Revuz
measures of $A_n$. Then there exists a subsequence $(n_k)$ satisfying
the condition that for q.e. $x\in E$,%
\[
\mathbf{P}_x \bigl(A_t^{n_k}\mbox{ converges
to zero uniformly on any compact} \bigr)=1.
\]
\end{lema}

%

\begin{lema}\label{liwrtcNu}
For every $u$ in $\mathcal{F}$ and $f$ in $\mathcal{F}_b$, there
exists a unique $w$ in
$\mathcal{F}$, such that:
\[
e\bigl(f*M^{u,c},M^h\bigr)=\mathcal{E}_1(w,h)
\qquad \forall h\in\mathcal{F}.
\]
\end{lema}

\begin{pf}
For $h\in\mathcal{F}$, $[e(f*M^{u,c},M^h)]^2\leq
e(f*M^{u,c})e(M^{h,c})\leq
e(f*M^{u,c})\tilde{\mathcal{E}}_1(h,h)$. Since $e(f*M^{u,c})<\infty
$, the
functional $h\rightarrow e(f*M^{u,c},M^h)$ is continuous.
\end{pf}

\begin{defi}\label{iwrtcNu}
For every $u$ in $\mathcal{F}$ and $f$ in $\mathcal{F}_b$, the
stochastic integral of
$f$ with respect to $^c\tilde{N}^u$ denoted by $\int_0^{\cdot}
f(X_s)\,\mathrm{d}^c\tilde{N}^u_s$ or by $f*{}^c\tilde{N}^u$ is defined by:
\[
\int_0^t f(X_s)
\,\mathrm{d}^c\tilde{N}^u_s
:=N^w_t-\int_0^t
w(X_s)\,\mathrm{d} s-\frac{1}{2}\bigl\langle
M^{f,c},M^{u,c}\bigr\rangle_t,\qquad t\geq0,
\]
where $w$ is the element of $\mathcal{F}$ associated to $(u,f)$ by
Lemma \ref
{liwrtcNu}.
\end{defi}

For any $u,v\in\mathcal{F}$, let $\mu^c_{\langle u,v\rangle }$ be the signed Revuz measure
associated to $\langle M^{u,c},M^{v,c}\rangle$. We have: $\frac
{1}{2}\mu^c_{\langle u,v\rangle }(E)=\tilde{\mathcal{E}}^{(c)}(u,v)$. For $f,h$ in
$\mathcal{F}_b$,
we have (Theorem 5.4 of \cite{K})
%
\begin{equation}
\label{l3.2.5FOT} \mathrm{d}\mu^c_{\langle u,hf\rangle }=f\,\mathrm{d}
\mu^c_{\langle u,h\rangle }+h\,\mathrm{d}\mu^c_{\langle u,f\rangle }.
\end{equation}

\begin{lema}\label{limiwrtcNu}
\mbox{}
\begin{enumerate}[(ii)]
\item[(i)] For $u$ in $\mathcal{F}$ and $f$ in $\mathcal{F}_b$, we
have $f*{}^c\tilde{N}^u\in\mathcal{N}_c^0$ and
%
\begin{equation}
\label{042607} \big\langle \Theta\bigl(f*{}^c\tilde{N}^u
\bigr),h\big\rangle =\big\langle \Theta\bigl({}^c\tilde {N}^u\bigr),fh\big\rangle
\qquad \mbox{for all }h\in\mathcal{F}_b.
\end{equation}
In particular, the integral is well defined in the following sense.
If $u,v\in\mathcal{F}$ are such that $^c\tilde{N}^u={}^c\tilde{N}^v$,
then for any $f\in\mathcal{F}_b$, $f*{}^c\tilde{N}^u=f*{}^c\tilde{N}^v$.

\item[(ii)] For $(u_n)$ a sequence of $\mathcal{F}$ $\tilde{\mathcal{E}
}_1$-converging to $u$, there exists a subsequence $(n_k)$ such that
for q.e. $x\in E$:
\[
\mathbf{P}_x\bigl(f*{}^c\tilde{N}^{u_{n_k}}
\mbox{ converges to }f*{}^c\tilde{N}^{u}\mbox{
uniformly on any compact}\bigr)=1.
\]
\end{enumerate}
\end{lema}

\begin{pf}
(i) $f*{}^c\tilde{N}^u\in\mathcal{N}_c^0$ because $|\mu^c_{\langle u,f\rangle }|(E)<\infty$.
Besides for any $h\in\mathcal{F}_b$,
\begin{eqnarray*}
\big\langle \Theta\bigl(f*{}^c\tilde{N}^u\bigr),h\big\rangle &=&-e
\bigl(f*M^{u,c},M^h\bigr)-\frac
{1}{2}\int
_E h(x)\,\mathrm{d}\mu^c_{\langle f,u\rangle }
\nonumber
\\
&=&-\frac{1}{2}\int_E f(x)\,\mathrm{d}
\mu^c_{\langle h,u\rangle }-\frac{1}{2}\int_E
h(x)\,\mathrm{d}\mu^c_{\langle f,u\rangle }.
\nonumber
\end{eqnarray*}
Then (\ref{042607}) is consequence of (\ref{lcNu}) and (\ref
{l3.2.5FOT}). The second statement is consequence of Lemma~\ref{T2.2N}.

(ii) Note that for any $u,v$ in $\mathcal{F}$, $f*{}^c\tilde
{N}^u-f*{}^c\tilde{N}^v=f*{}^c\tilde{N}^{u-v}$. Thus we
need only to show that if $(u_n)$ converges to $0$ and $f\in\mathcal{F}_b$,
there exists a subsequence $(n_k)$ such that for q.e. $x\in E$:
\[
\mathbf{P}_x\bigl(f*{}^c\tilde{N}^{u_{n_k}}
\mbox{ converges to }0\mbox{ uniformly on any compact}\bigr)=1.
\]
For each $n$, let $w_n$ be the function associated to $(f,u_n)$ by
Lemma \ref{liwrtcNu}. Then for any $h\in\mathcal{F}$ we have:
$\mathcal{E}_1(w_n,h)^2\leq\|f^2\|_{\infty}\mathcal{E}_1(h,h)\mathcal
{E}_1(u_n,u_n)$. In
particular, choosing $h=w_n$, one obtains:
\[
\mathcal{E}_1(w_n,w_n)\leq\big\|f^2
\big\|_{\infty}\mathcal {E}_1(u_n,u_n)
\rightarrow0\qquad \mbox{as }n\rightarrow\infty.
\]
It follows from Lemma \ref{l5.1.2FOT} that there exists a subsequence
$(n_k)$ such that $\mathbf{P}_x$-a.e. for q.e. $x\in E$,
$N_t^{w_{n_k}}-\int_0^t w_{n_k}(X_s)\,\mathrm{d}s$ converges to $0$ uniformly on compacts.

Besides: $\mu^c_{\langle u_n\rangle }(E)=\tilde{\mathcal{E}}^{(c)}(u_n,u_n)$, which
converges to $0$. Hence, by Lemma \ref{convCAF}, there exists a
subsequence $(n_k)$ such that
\[
\big|\bigl\langle M^{u_{n_k},c},M^{f,c}\bigr\rangle\big|\leq\bigl\langle
M^{f,c}\bigr\rangle^{1/2}\bigl\langle M^{u_{n_k},c}\bigr
\rangle^{1/2}
\]
converges to $0$ on compacts $\mathbf{P}_x$-a.e. for q.e. $x\in E$.
\end{pf}

\subsection{\texorpdfstring{Stochastic integration with respect to $^j\tilde{N}^u$}
{Stochastic integration with respect to jNu}}\label{subsectionjNu}

Denote by $(N,H)$ the L\'{e}vy system of $X$. Let $\hat{X}$ be the
Markov process properly associated to the Dirichlet form $\hat
{\mathcal{E}
}(u,v):=\mathcal{E}(v,u)$, $u,v\in\mathcal{F}$ and $(\hat
{N},H_{\hat{X}})$ its L\'
evy system. Let $\nu_H$ be the Revuz measure associated to $H$ and let
$\nu_{\hat{H}}$ be the Revuz measure associated to $H_{\hat{X}}$ and
$\hat{H}$ be the PCAF of $X$ associated to $\nu_{\hat{H}}$ by the
Revuz correspondence. Let $J$, $\hat{J}$ and $\tilde{J}$ denote
respectively the jumping measure of $\mathcal{E}$, $\hat{\mathcal
{E}}$ and $\tilde{\mathcal{E}
}$, that is, $J(\mathrm{d}y,\mathrm{d}x)=\frac{1}{2}N(x,\mathrm
{d}y)\nu_H(\mathrm{d}x)$,
$\hat{J}(\mathrm{d}y,\mathrm{d}x)=\frac{1}{2}\hat{N}(x,\mathrm
{d}y)\nu_{\hat{H}}(\mathrm{d}
x)$ and $\tilde{J}(\mathrm{d}x,\mathrm{d}y)=\frac{1}{2}[J(\mathrm
{d}x,\mathrm{d}y)+\hat
{J}(\mathrm{d}x,\mathrm{d}y)]$. It is known that $\hat{J}(\mathrm
{d}y,\mathrm{d}x)=J(\mathrm{d}x,\mathrm{d}y)$.

We will use the following notations:\vspace*{-1pt}
\begin{eqnarray*}
\mathbf{N}(\mathrm{d}y,\mathrm{d}s)&:=&N(X_s,\mathrm{d}y)\,\mathrm
{d}H_s,\quad \mbox{and }
\nonumber
\\
\tilde{\mathbf{N}}(\mathrm{d}y,\mathrm{d}s)&:=&
\tfrac
{1}{2}\bigl(N(X_s,\mathrm{d}y)\,\mathrm{d}H_s+
\hat {N}(X_s,\mathrm{d}y)\,\mathrm{d}\hat{H}_s\bigr).
\nonumber
\end{eqnarray*}
For any $u\in\mathcal{F}$, denote by $M^{u,j}$ the jump part of $M^u$ (see
page 213 of \cite{FOT} for the definition), this is an element of $\M
$ and for all $h\in\mathcal{F}$, $e(M^{u,j},M^h)=\tilde{\mathcal
{E}}^{(j)}(u,h)$.
With the same arguments used to show Lemma \ref{liwrtcNu}, we can
obtain the following lemma.

\begin{lema}\label{liwrtjNu}
For every $u$ in $\mathcal{F}$ and $f$ in $\mathcal{F}_b$, there
exists a unique $w$ in
$\mathcal{F}$, such that:
\[
e\bigl(f*M^{u,j},M^h\bigr)=\mathcal{E}_1(w,h)\qquad
\forall h\in\mathcal{F}.
\]
\end{lema}

\begin{defi}\label{diwrtjNu}
For every $u$ in $\mathcal{F}$ and $f$ in $\mathcal{F}_b$, the
stochastic integral of
$f$ with respect to $^j\tilde{N}^u$ denoted by $\int_0^{\cdot}
f(X_s)\,\mathrm{d}^j\tilde{N}^u_s$ or by $f*{}^j\tilde{N}^u$ is defined by:
\begin{eqnarray*}
\int_0^t f(X_s)\,
\mathrm{d}^j\tilde{N}^u_s
&:=&N^w_t-\int_0^t
w(X_s)\,\mathrm{d} s\\
&&{}-\frac{1}{2}\int_0^t
\int_E\bigl[f(x)-f(X_s)\bigr]
\bigl[u(x)-u(X_s)\bigr]\tilde{\mathbf{N} }(\mathrm{d}x,\mathrm{d}s),\qquad t
\geq0,
\end{eqnarray*}
where $w$ is the element of $\mathcal{F}$ associated to $(u,f)$ by
Lemma \ref
{liwrtjNu}.
\end{defi}

\begin{lema}\label{limiwrtjNu}
\mbox{}
\begin{enumerate}[(ii)]
\item[(i)]For $u\in\mathcal{F}$ and $f\in\mathcal{F}_b$, $f*{}^j\tilde{N}^u$
belongs to $\mathcal{N}_c^0$ and for $h$ in $\mathcal{F}_b$:
%
\begin{equation}\label{072607}
\big\langle \Theta\bigl(f*{}^j\tilde{N}^u\bigr),h\big\rangle =\big\langle \Theta
\bigl({}^j\tilde {N}^u\bigr),fh\big\rangle .
\end{equation}
In particular, the integral is well defined in the following sense.
If $u,v$ in $\mathcal{F}$ are such that $^j\tilde{N}^u={}^j\tilde
{N}^v$, then for any $f$ in $\mathcal{F}_b$: $f*{}^j\tilde
{N}^u=f*{}^j\tilde{N}^v$.
\item[(ii)] If $(u_n)$ is $\tilde{\mathcal{E}}_1$-converging to $u$, there
exists a subsequence $(n_k)$ such that for q.e. $x\in E$:
\[
\mathbf{P}_x\bigl(f*{}^j\tilde{N}^{u_{n_k}}
\mbox{ converges to }f*{}^j\tilde{N}^{u}\mbox{
uniformly on any compact}\bigr)=1.
\]

\end{enumerate}
\end{lema}

\begin{pf}
The proof of (ii) is similar to the proof of (ii) of Lemma \ref
{limiwrtcNu}. We prove (i). Clearly, $f*{}^j\tilde{N}^u$ belongs
to $\mathcal{N}_c^0$ and for any $h\in\mathcal{F}_b$:
\begin{eqnarray*}
&&\big\langle \Theta\bigl(f*{}^j\tilde{N}^u,h\bigr)\big\rangle
\nonumber
\\
&&\quad =e\bigl(f*M^{u,j},M^h\bigr)-\int_{E\times E\backslash\delta
}h(y)
\bigl[f(x)-f(y)\bigr] \bigl[u(x)-u(y)\bigr]\tilde{J}(\mathrm{d}x,\mathrm {d}y)
\nonumber
\\
&&\quad =-\int_{E\times E\backslash\delta}\bigl[f(y)\bigl\{h(x)-h(y)\bigr\}+h(y)\bigl\{
f(x)-f(y)\bigr\}\bigr]
\bigl[u(x)-u(y)\bigr]\tilde{J}(\mathrm{d}x,\mathrm{d}y),
\nonumber
\end{eqnarray*}
where $\delta:=\{(x,x)\dvt x\in E\}$. Using the symmetry of $\tilde{J}$
and the fact that $J(\mathrm{d}x,\mathrm{d}y)+J(\mathrm{d}y,\mathrm
{d}x)=2\tilde{J}(\mathrm{d}x,\mathrm{d}
y)$, one proves that the right-hand\vadjust{\goodbreak} side of the above equation
coincides with:
\[
-\int_{E\times E\backslash\delta
}\bigl[h(x)f(x)-h(y)f(y)\bigr] \bigl[u(x)-u(y)
\bigr]J(\mathrm{d} x,\mathrm{d}y)=-\tilde{\mathcal{E}}^{(j)}(u,hf).
\]
Then \ref{072607} is consequence of \ref{lcNu}. The second statement
can be shown in the same way that its analogous in Lemma \ref{limiwrtcNu}(i).
\end{pf}

\subsection{Stochastic integration with respect to $N^u$}\label{subsectionNu}

In view of the decomposition (\ref{032907}), we can define the
stochastic integral of $f(X)$ with respect to $N^u$ for $f\in\mathcal{F}_b$
and $u\in\mathcal{F}$ as follows.

\begin{defi}
For any $u\in\mathcal{F}$ and $f\in\mathcal{F}_b$, the stochastic
integral of $f(X)$
with respect to $N^u$ denoted by $f*N^{u}$ or by $\int_0^t
f(X_s)\,\mathrm{d}
N_s^u$ is defined by
\begin{eqnarray*}
\int_0^t f(X_s)\,
\mathrm{d}N_s^u&:=&\int_0^t
f(X_s)\,\mathrm{d}{}^c\tilde {N}^{u^*}_s+
\int_0^t f(X_s)\,\mathrm{d}
{}^j\tilde {N}^{u^*}_s
\nonumber
\\
&&{}-\int_0^t f(X_s)u^*(X_s)
\,\mathrm{d}\tilde{K}_s+\int_0^t
f(X_s) \bigl(u(X_s)-u^*(X_s)\bigr)
\,\mathrm{d}s,
\nonumber
\end{eqnarray*}
where the first two integrals are in the sense of the definitions (\ref
{iwrtcNu}) and (\ref{diwrtjNu}), respectively, and the others integrals
are Lebesgue--Stieltjes integrals.
\end{defi}

It is clear that for any $u$ in $\mathcal{F}$ and $f$ in $\mathcal
{F}_b$, the
stochastic integral $f*N^u$ belongs to $\mathcal{N}_c^0$ and in view
of (\ref
{042607}) and (\ref{072607}), we have
\[
\big\langle \Theta\bigl(f*N^u\bigr),h\big\rangle =\big\langle \Theta\bigl(N^u\bigr),fh\big\rangle =-
\mathcal{E}(u,fh)\qquad \mbox{for all }h\in\mathcal{F}_b.
\nonumber
\]
Let $(u_n)$ be a sequence on $\mathcal{F}$ $\tilde{\mathcal
{E}}_1$-converging to $u\in
\mathcal{F}$, it follows from (\ref{020311}) that $(u_n^*)$ $\tilde
{\mathcal{E}
}_1$-converges to $u^*$ then thanks to Lemma \ref{l5.1.2FOT}, Lemma
\ref{limiwrtcNu}(ii) and Lemma \ref{limiwrtjNu}(ii) we have the
following lemma.

\begin{lema}\label{lema3.11}
Let $f$ be a function in $\mathcal{F}_b$ and $(u_n)$ a sequence on
$\mathcal{F}$
$\tilde{\mathcal{E}}_1$-converging to $u\in\mathcal{F}$. Then there
exists a
subsequence $(n_k)$ such that for q.e. $x\in E$:
\[
\mathbf{P}_x\bigl(f*N^{u_{n_k}}\mbox{ converges uniformly on
any compact to }f*N^u\bigr)=1.
\]
\end{lema}

Let $A$ be the CAF defined by $A_t:=N_t^u-\int_0^t u(X_s)\,\mathrm
{d}s$, for
an element $u$ of $\mathcal{F}$ and let $f$ be a function in $\mathcal
{F}_b$, the
stochastic integral of $f(X)$ with respect to $A$ is defined by:
\[
f*A_t=\int_0^tf(X_s)
\,\mathrm{d}A_s:=\int_0^t
f(X_s)\,\mathrm {d}N_s^u-\int
_0^t f(X_s)u(X_s)
\,\mathrm{d}s.
\]

\begin{lema}\label{011211}
Let $u$ and $v$ be elements of $\mathcal{F}$, $f$ and $g$ elements of
$\mathcal{F}_b$
and $G$ a nearly Borel finely open set. Set $A_t:=N_t^u-\int_0^t
u(X_s)\,\mathrm{d}s$ and $B_t:=N_t^v-\int_0^t v(X_s)\,\mathrm{d}s$.
Suppose that
$f(x)=g(x)$ for q.e. $x\in G$ and $\mathbf{P}_x (A_t=B_t,\mbox{ for
any }t<\sigma_{E\setminus G} )=1$ for q.e. $x\in E$.
Then $\mathbf{P}_x (f*A_t=\allowbreak g*B_t,\mbox{ for any }t<\sigma_{E\setminus
G} )=1$ for q.e. $x\in E$.
\end{lema}

\begin{pf}
It follows from Lemma \ref{T2.2N} that for any $h\in\mathcal{F}_G$
$\langle \Theta
(A),h\rangle =\langle \Theta(B),h\rangle $ then $\langle \Theta(f*A),h\rangle =\langle \Theta(A),fh\rangle =\langle \Theta
(B),gh\rangle =\langle \Theta(g*B),h\rangle $. We conclude thanks to Lemma \ref{T2.2N}.
\end{pf}

A function $f$ belongs to $\mathcal{F}_{\mathit{loc}}$ if there exists a sequence
$(f_n)$ of $\mathcal{F}$ and a nest of nearly Borel finely open sets $(G_n)$
such that $f(x)=f_n(x)$ for q.e. $x\in G_n$. In fact the sequence
$(f_n)$ can be taken in $\mathcal{F}_b$ (see Lemma 3.1 in \cite{CFKZ}).

With the above lemma and Theorem \ref{teore1.1walsh}, we can define
the stochastic integral of $f(X_s)$ with respect to $C$ for any $f\in
\mathcal{F}_{\mathit{loc}}$ and $C$ in $\mathcal{N}_{f\mbox{-}\mathit{loc}}$.

\begin{defi}\label{intsto}
Let $C$ be an element of $\mathcal{N}_{f\mbox{-}\mathit{loc}}$ and $f$ in
$\mathcal{F}_{\mathit{loc}}$.
Let $(G_n)$ and $(u_n)$ be the sequences of the conclusion of Theorem
\ref{teore1.1walsh} and $(f_n)\subset\mathcal{F}_b$ such that $f(x)=f_n(x)$
for q.e. $x\in G_n$. Set $C^n_t:=N_t^{u_n}-\int_0^t u_n(X_s)\,\mathrm{d}s$.
Then if $\sigma:=\lim_{n\rightarrow\infty}\sigma_{E\setminus
G_n}$, we define the stochastic integral of $f$ with respect to $A$ and
denoted by $f*C_t,t\geq0$ or by $\int_0^tf(X_s)\,\mathrm{d}C_s,t\geq
0$ as
the following local CAF:
\[
\label{sig} f*C_t:= \cases{ %
f_n*C^n_t&\quad  for $t<
\sigma_{E\setminus G_n}$,
\cr
0&\quad  for $t\geq\sigma$. }
\]
\end{defi}

\begin{nota}\label{nota101}
\begin{enumerate}[(iii)]
\item[(i)] It follows from Lemma \ref{011211} that the above
definition makes sense and not depend of the sequences $C^{n}$, $(f_n)$
nor $(G_n)$.
\item[(ii)] Any PCAF belongs to $\mathcal{N}_{f\mbox{-}\mathit{loc}}$, then
with the
notation of the above definition, $f*C$ belongs to $(\mathcal
{N}_{f\mbox
{-}\mathit{loc}})_{f\mbox{-}\mathit{loc}}=\mathcal{N}_{f\mbox{-}\mathit{loc}}$.
\item[(iii)] Let $\varphi\dvtx \mathbb{R}\rightarrow\mathbb{R}$ be a
function admitting
a continuous derivative. For any $n\in\mathbb{N}$ let $\varphi_n$ be a
function admitting a bounded continuous derivative such that $\varphi_n(x)=\varphi(x)$ if $|x|<n$. We know that $\varphi_n(u)-\varphi(0)$
belongs to $\mathcal{F}$ for any $u\in\mathcal{F}$ and if we set
$G_n:=\{x\dvt |u(x)|<n\}
$, $(G_n)$ is a nest of finely open sets. Since $u$ is quasi continuous
in the strict sense (i.e., $u(X_t)\rightarrow u(X_{\zeta-})\in
\mathbb{R}$
as $t\uparrow\zeta$), $\sigma_{E\setminus G_n}\uparrow\infty$
$\mathbf{P}_x$-a.s. for q.e. $x\in E$. Therefore for any $v\in
\mathcal{F}$ the
stochastic integral $\varphi(u)*N^v=[\varphi(u)-\varphi
(0)]*N^v+\varphi(0)N^v$ is a finite CAF. This hold also for $\varphi
(u_1,\ldots,u_k)*N^v$ for any $u_1,\ldots,u_k$ and $v$ in $\mathcal
{F}$ and
$\varphi\in C^1(\mathbb{R}^k)$.
\end{enumerate}
\end{nota}

\section{\texorpdfstring{Proof of Theorem \protect\ref{integralestocastica}}{Proof of Theorem 1.1}}\label{prueba1}

In this section, we show that for $f\in\mathcal{F}_{\mathit{loc}}$ and $C\in
\mathcal{N}_{f\mbox{-}\mathit{loc}}$, the additive functional $f*C$ built in the precedent
section satisfies the conclusion of Theorem \ref{integralestocastica}.

Without loss of generality, we take in this section $\Omega$ to be the
canonical path space $D([0,\infty)\rightarrow\infty)\rightarrow
E_{\partial}$ of c\`{a}dl\`{a}g functions from $[0,\infty)$ to
$E_{\partial}$ for which $w(t)=\partial$ for all $t\geq\zeta(\omega
):=\inf\{s\geq0\dvt w(s)=\partial\}$.

Given $\omega\in\{\omega\in\Omega\dvt t<\zeta(\omega)\}$, the
operator $r_t$ is defined by:
\[
r_t(\omega) (s):= \cases{ %
 \omega
\bigl((t-s)-\bigr)& \quad if $0\leq s< t$,
\cr
\omega(0)&\quad  if $s\geq t$. }
\]
We denote by $\{\hat{P}_x,x\in E\}$ the law of $\hat{X}$, the dual
process of $X$. The following lemma can be established using the same
arguments as Lemma 5.7.1 in \cite{FOT}.

\begin{lema}\label{lema012111}
For positive $t$ and every $\mathcal{F}_t$-measurable set $\Gamma$,
\[
P_m\bigl(r_t^{-1}\Gamma;t<\zeta\bigr)=
\hat{P}_m(\Gamma,t<\zeta).
\]
\end{lema}

\begin{lema}\label{01221111}
For any $u$ in $\mathcal{F}$, there exists a unique $\hat{u}\in
\mathcal{F}$ such that
$\mathcal{E}_1(\hat{u},h)=\mathcal{E}_1(h,u)$ for any $h$ in
$\mathcal{F}$. If we set:
\begin{eqnarray*}
\hat{N}_t^u&:=&N_t^{\hat{u}}+\int
_0^t \bigl(u(X_s)-
\hat{u}(X_s)\bigr)\,\mathrm{d} s,\qquad  t\geq0,
\nonumber
\\
\hat{M}_t^u&:=&u(X_t)-u(X_0)-
\hat{N}_t^u,
\nonumber
\end{eqnarray*}
then under $(\hat{P}_x,x\in E)$, $\hat{N}^u$ and $\hat{M}^u$ are,
respectively, the CAF of zero energy and the MAF of finite energy of the
Fukushima decomposition for $u(X_t)-u(X_0)$, $t\geq0$.
\end{lema}

\begin{pf}
For any $n\in\mathbb{N}$ set $\hat{f}_n:=n(u-n\hat{R}_{n+1}u)$. The
constant $K$ was introduced in (\ref{020311}). For any $h$ in
$\mathcal{F}$
and $n,m$ in $\mathbb{N}$:
\begin{eqnarray*}
\mathcal{E}_1\bigl(R_1(\hat{f}_n-
\hat{f}_m),h\bigr)&=&\mathcal{E}_1\bigl(h,\hat
{R}_1(\hat{f}_n-\hat {f}_m)\bigr)\\
&\leq& K
\bigl(\mathcal{E}_1(h,h)\bigr)^{1/2}\bigl(
\mathcal{E}_1\bigl(\hat{R}_1(\hat {f}_n-\hat
{f}_m),\hat{R}_1(\hat{f}_n-
\hat{f}_m)\bigr)\bigr)^{1/2}.
\end{eqnarray*}
In particular, if $h=R_1(\hat{f}_n-\hat{f}_m)$ we obtain:
\[
\mathcal{E}_1\bigl(R_1(\hat{f}_n-
\hat{f}_m),R_1(\hat{f}_n-\hat
{f}_m)\bigr)\leq K^2 \mathcal{E}_1\bigl(
\hat{R}_1(\hat{f}_n-\hat{f}_m),
\hat{R}_1(\hat{f}_n-\hat{f}_m)\bigr).
\]
It is known that the right-hand side of the above equation tends to $0$
as $n,m$ tends to infinity (see Theorem I.2.13 in \cite{MaR}) then
there exists $\hat{u}$ in $\mathcal{F}$ such that $R_1 \hat{f}_n$ converges
to $\hat{u}$ with respect to the $\tilde{\mathcal{E}}_1$-norm.
Besides, for
any $h$ in $\mathcal{F}\dvt\mathcal{E}_1(\hat{u},h)=\lim\mathcal
{E}_1(R_1\hat{f}_n,h)=\lim
\mathcal{E}_1(h,\hat{R}_1\hat{f}_n)=\mathcal{E}_1(h,u)$.

Let $A_t$ be the CAF of zero energy of the Fukusmima decomposition of
$u(X_t)-u(X_0)$ with respect to $\hat{P}_x,x\in E$. By taking a
subsequence if necessary, we have $\hat{P}_x$-a.e. for q.e. $x\in E$:
For all $t\geq0$
\begin{eqnarray*}
A_t&=&\lim_{n\rightarrow\infty}\int_0^t
\bigl[u(X_s)-\hat {f}_n(X_s)\bigr]\,\mathrm{d}
s
\nonumber
\\
&=&\lim_{n\rightarrow\infty}\int_0^t \bigl[
\hat{u}(X_s)-\hat {f}_n(X_s)\bigr]
\,\mathrm{d}s+\int_0^t\bigl[u(X_s)-
\hat{u}(X_s)\bigr]\,\mathrm {d}s
\nonumber
\\
&=&\hat{N}_t^u.
\nonumber
\end{eqnarray*}
\upqed\end{pf}

Clearly, $\hat{N}^u$ belongs to $\mathcal{N}_c^0$ and:
%
\begin{equation}\label{022522}
\big\langle \Theta\bigl(\hat{N}^u\bigr),h\big\rangle =-\mathcal{E}(h,u),\qquad  h\in
\mathcal {F}.
\end{equation}

\begin{lema}\label{minuit}
Let $\hat{A}$ be a PCAF with respect to $(\hat{P}_x,x\in E)$ and with
Revuz measure $\mu$. Then under $(P_x,x\in E)$, $\hat{A}$ is the PCAF
with Revuz measure $\mu$.\vadjust{\goodbreak}
\end{lema}

\begin{pf}
We suppose without loss of generality that $\mu\in S_0$. Let $u=\hat
{U}_1\mu$ and $v:=U_1\mu$ be the 1-potentials of $\mu$ with respect
to $\mathcal{E}$ and $\hat{\mathcal{E}}$, respectively. Let $A$ be
the PCAF with
respect to $(P_x,x\in E)$ and with Revuz measure $\mu$. For any $h$ in
$\mathcal{F}$, $\mathcal{E}_1(h,u)=\mathcal{E}_1(v,h),$ then with
the notation of Lemma \ref
{lema012111}, $v=\hat{u}$. It follows from Theorem \ref{T5.4.2FOT}
that $P_x$-a.e. for q.e. $x\in E$:
\begin{eqnarray*}
A_t&=&-N_t^v+\int_0^t
v(X_s)\,\mathrm{d}s=-N_t^v-\int
_0^t \bigl[u(X_s)-v(X_s)
\bigr]\,\mathrm{d} s+\int u(X_s)\,\mathrm{d}s
\nonumber
\\
&=&-\hat{N}^u+\int_0^tu(X_s)
\,\mathrm{d}s =\hat{A_t}.
\nonumber
\end{eqnarray*}
\upqed\end{pf}

The following lemma can be found in \cite{Fi2}, Lemma 3.21, for
symmetric diffusions.

\begin{lema}\label{lema22211}
Let $u$ be in $\mathcal{F}$. For any $t\leq T$ we have $P_m$-a.e. on
$\{T<\zeta
\}$:
\begin{eqnarray*}
{\hat{N}^u_t} \circ r_T&=&{
\hat{N}^u_T}-{\hat{N}^u_{T-t}}\qquad \mbox{and}
\nonumber
\\
{\hat{M}^u_t} \circ r_T&=&{
\hat{M}^u_T} \circ r_T-{
\hat{M}^u_{T-t}} \circ r_{T-t}.
\nonumber
\end{eqnarray*}
\end{lema}

\begin{pf}
Define $\hat{u}$ and $(\hat{f_n})$ as in Lemma \ref{01221111}. $\hat
{P}_m$-a.e. and by taking subsequences if necessary we have: $\hat
{N}_t^u=\lim_{n\rightarrow\infty}\int_0^t[u(X_s)-\hat
{f}_n(X_s)]\,\mathrm{d}s$, thus in view of Lemma \ref{lema012111}, we have
$P_m$-a.e. on $\{T<\zeta\}$:
\begin{eqnarray*}
\hat{N}_t^u\circ r_T&=&\lim_{n\rightarrow\infty}
\int_0^t\bigl[u(X_s)-
\hat{f}_n(X_s)\bigr]\,\mathrm{d}s\circ r_T
\nonumber
\\
&=&\lim_{n\rightarrow\infty}\int_0^T
\bigl[u(X_s)-\hat{f}_n(X_s)\bigr]\,\mathrm{d}
s-\lim_{n\rightarrow\infty}\int_0^{T-t}
\bigl[u(X_s)-\hat {f}_n(X_s)\bigr]\,\mathrm{d}
s
\nonumber
\\
&=&\hat{N}^u_T-
\hat{N}^u_{T-t}.
\nonumber
\end{eqnarray*}
The second equality can be shown with easy computations using the first one.
\end{pf}

\begin{nota}\label{nota29}
The first equality in Lema \ref{lema22211} is in fact true for $N^u$
and therefore for the elements in $\mathcal{N}_{f\mbox{-}\mathit{loc}}$, in particular
for any PCAF.
\end{nota}

Similarly to \cite{CFKZ}, the proof of Theorem \ref
{integralestocastica} is based in an extension of the Lyons and Zheng
decomposition \cite{LZ}, that is, in a representation of $N^u$ using
forward and backward MAF. We recall that for $u$ in $\mathcal{F}$,
$u^*$ was
defined as the unique element of $\mathcal{F}$ satisfying~(\ref{eq2}).

\begin{lema}
Let $u$ be in $\mathcal{F}$ and $T$ in $\mathbb{R}_+$. Set $v:=u^*$.
Then we have
$\mathbf{P}_m$-a.e. on $\{T<\zeta\}$:
\begin{eqnarray}\label{012511}
N^u_t&=&-\frac{1}{2}\bigl(M^v_t+
\hat{M}^v_t\circ r_t\bigr)+\frac
{1}{2}
\bigl(v(X_t)-v(X_{t-})\bigr)
\nonumber\\[-8pt]\\[-8pt]
&&{}+\int_0^t
\bigl[u(X_s)-v(X_s)\bigr]\,\mathrm{d}s,\qquad  t\leq T.
\nonumber
\end{eqnarray}
\end{lema}

\begin{pf}
In view of Lemma \ref{lema22211}, the right-hand side of (\ref
{012511}) coincides $\mathbf{P}_m$-a.e. on $\{T<\zeta\}$ with~$A$, where
for all $t\leq T$, $A_t:=\frac{1}{2}(N_t^v+\hat{N}_t^v)+\int_0^t[u(X_s)-v(X_s)]\,\mathrm{d}s$. It follows from (\ref{022522}) that
$\langle \Theta(A),h\rangle =-\tilde{\mathcal{E}}(v,h)+(u,h)-(v,h)=-\mathcal
{E}(u,h)$, for all
$h\in\mathcal{F}$. Now, \ref{012511} is consequence of Lemma \ref{T2.2N}.
\end{pf}

\begin{lema}\label{lemavqz}
Let $(N^{\ell})_{\ell\in\mathbb{N}}$ be a sequence of elements of
$\in\mathcal{N}_{c,f\mbox{-}\mathit{loc}}$ and let $(\Pi_n)$ be a sequence of partitions
tending to the identity. Then there exists a subsequence $(\Pi_{n_j})$
of $(\Pi_n)$ such that $\mathbf{P}_x$-a.s. for $m$-a.e. $x$ in E we have:
For all $\ell\in\mathbb{N}$,
\[
\sum_{k=0}^{p_{n_j}-1}\bigl[N^{\ell}(t
\wedge t_{n_j,k+1})-N(t\wedge t_{n_j,k})\bigr]^2
\]
converges to zero as $n\rightarrow\infty$, uniformly in any compact
of $[0,\infty)$.
\end{lema}

\begin{pf}
Let $g$ be a function belongs to $L^1(E,m)$ such that $0<g(x)\leq1$
for all $x\in E$. For any $t\in\mathbb{R}_+$ set $a_n(t):=\sup\{
t_{n,k}\dvt t_{n,k}<t\}$. For all $n,\ell,K\in\mathbb{N}$, $\eta>0$:
%
\begin{eqnarray}\label {012930}
&&\mathbf{P}_{g\cdot m} \Biggl(\sup_{t\leq K}\sum
_{k=0}^{p_n-1}\bigl[N^{\ell
}(t\wedge
t_{n,k+1})-N^{\ell}(t\wedge t_{n,k})\bigr]^2>
\eta \Biggr)\nonumber
\\
&&\quad \leq\frac{2}{\eta}
\mathbf{E}_m \Biggl(\sum_{k=0}^{p_n-1}
\bigl[N^{\ell
}(K\wedge t_{n,k+1})-N^{\ell}(K\wedge
t_{n,k})\bigr]^2 \Biggr)
\\
&&\qquad {}+\mathbf{P}_{g\cdot m} \biggl(\sup_{t\leq K}\bigl[N^{\ell}(t)-N^{\ell
}
\bigl(a_n(t)\bigr)\bigr]^2>\frac{\eta}{2} \biggr).
\nonumber
\end{eqnarray}
Since $N^{\ell}\in\mathcal{N}_c$, the last term in the above equation
converges to zero as $n\rightarrow\infty$. (See (5.2.14) in \cite
{FOT}.) For all $n,\ell,K\in\mathbb{N}$ and $\eta>0$ let $\alpha
(n,\ell
,K,\eta)$ be the left-hand side of (\ref{012930}). Then for all $\eta>0$:
\[
\beta(n,\eta):=\sum_{\ell,K\in\mathbb{N}}\frac{1}{\ell^2K^2}\alpha
(n,\ell,K,\eta)\rightarrow0\qquad \mbox{as }n\rightarrow\infty.
\]
For any $j\in\mathbb{N}$ take $n_j$ such that $\beta(n_j,j^{-1})\leq
j^{-2}$. Then for any $j,\ell,K\in\mathbb{N}$, $\alpha(n_j,\ell
,K,\break j^{-1})\leq\ell^2K^2j^{-2}$ therefore, it follows from
Borel--Cantelli that for all $\ell,K$, $\sum_{k=0}^{p_{n_j}-1}[N^{\ell
}(t\wedge t_{n_j,k+1})-N^{\ell}(t\wedge t_{n_j,k})]^2$ converges to
zero as $n\rightarrow\infty$ uniformly on $[0,K]$ $\mathbf{P}_{g\cdot m}$-a.s.
\end{pf}

\begin{pf*}{Proof of Theorem \ref{integralestocastica}}
As usually, the uniqueness in the theorem is the following sense: two
local AF $A$, $B$ are equivalent if $\mathbf{P}_x(A_t=B_t,t<\zeta)=1$ for
q.e. $x\in E$. Evidently if $I^1$ and $I^2$ are two local AF satisfying
the conclusion of the theorem then, $\mathbf{P}_x(I^1_t=I^2_t,t<\zeta)=1$
for $m$-a.e. $x\in E$. We can show that this hold for q.e. $x\in E$
using an argument of the proof of Proposition~4.6 in \cite{CFKZ}.

Now we shall proof that the stochastic integral $\int_0^t
f(X_s)\,\mathrm{d}
C_s$ of the precedent section satisfies the conclusion of theorem. Let
$\{u_n\}$, $\{f_n\}$ and $\{G_n\}$ be a sequence\vadjust{\goodbreak} of $\mathcal{F}$,
$\mathcal{F}_b$ and
$\Xi$, respectively, such that $C_t=C^n_t:=N_t^{u_n}-\int_0^t
u_n(X_s)\,\mathrm{d}s$ on $\llbracket0,\tau_{G_n}\llbracket$ $\mathbf
{P}_m$-a.e.
and $f=f_n$
q.e. on $G_n$ (Theorem~\ref{teore1.1walsh}). For each $n$ set
$v_n:=u_n^*$. In order to simplify the notation let $M^n$ be $M^{v_n}$
and in the same way define $\hat{N}^n$ and $\hat{M}^n$. For all
$t\leq T$ set $a_n(t)=\sup\{t_{n,k}\dvt t_{n,k}<t\}$ and set:
\begin{eqnarray*}
X_t^n&:=&\sum_{k=0}^{p_n-1}X(t_{n,k})1_{\{t_{n,k}<t\leq t_{n,k+1}\}
}
,\qquad Y_t^n:=\sum_{k=0}^{p_n-1}X(t_{n,k+1})1_{\{t_{n,k}\leq t<
t_{n,k+1}\}}
\quad \mbox{and }
\nonumber
\\[-2pt]
Z_T^n(t)&:=&Y^n_{T-t}\circ
r_T=\sum_{k=0}^{p_n-1}X(T-t_{n,k+1})1_{\{
T-t_{n,k+1}<t\leq T-t_{n,k}\}}.
\nonumber
\end{eqnarray*}

In view of (\ref{012511}) and Lemma \ref{lema22211}, we have for any
$\ell\in\mathbb{N}$:
%
\begin{eqnarray}\label{aproxsi}
&&\sum_{k=0}^{p_n-1}f_{\ell}
\bigl(X(t_{n,k})\bigr)\bigl[C^{\ell}(t_{n,k+1}\wedge
t)-C^{\ell}(t_{n,k}\wedge t)\bigr]\nonumber
\\[-2pt]
&&\quad =-\frac{1}{2}\int_0^t f_{\ell}
\bigl(X_s^n\bigr)\,\mathrm{d}M_s^{\ell}-
\frac
{1}{2}\int_t^T f_{\ell}
\bigl(Z_T^n(s)\bigr)\,\mathrm{d}\hat{M}^{\ell}_s
\circ r_T-\int_0^t f_{\ell}
\bigl(X_s^n\bigr)v_{\ell}(X_s)
\,\mathrm{d}s
\nonumber
\\[-2pt]
&&\qquad {}-\frac{1}{2}\sum_{k=0}^{p_n-1}
\bigl[f_{\ell}\bigl(X(t_{n,k+1})\bigr)-f_{\ell
}
\bigl(X(t_{n,k})\bigr)\bigr] \bigl[v_{\ell}\bigl(X(t_{n,k+1}
\wedge t)\bigr)-v_{\ell
}\bigl(X(t_{n,k}\wedge t)\bigr)\bigr]
\\[-2pt]
&&\qquad {}+\frac{1}{2}f_{\ell}\bigl(X\bigl(a_n(t)\bigr)\bigr)
\bigl(v_{\ell}(X_t)-v_{\ell
}(X_{t-})\bigr)
\nonumber
\\[-2pt]
&&\qquad {}-\frac{1}{2}\sum_{k=0}^{p_n-1}
\bigl[f_{\ell}\bigl(X(t_{n,k+1})\bigr)-f_{\ell
}
\bigl(X(t_{n,k})\bigr)\bigr] \bigl[\hat{N}^{\ell}
\bigl(X(t_{n,k+1}\wedge t)\bigr)-\hat{N}^{\ell
}
\bigl(X(t_{n,k}\wedge t)\bigr)\bigr].
\nonumber
\end{eqnarray}
In view of Lemmas \ref{lema012111} and \ref{lemavqz}, the right-hand
side of (\ref{aproxsi}) converges in $\mathbf{P}_{g\cdot m}$-measure on \mbox{$\{
T<\zeta\}$} to:
\begin{eqnarray}\label{expsi}
I^{\ell}_T(t)&:=&-\frac{1}{2}\int
_0^tf_{\ell}(X_s)\,\mathrm
{d}M_s^{\ell
}-\frac{1}{2}\int_t^T
f_{\ell}(X_s)\,\mathrm{d}\hat{M}_s^{\ell
}
\circ r_T-\int_0^t
f_{\ell}(X_s)v_{\ell}(X_s)\,\mathrm{d}s
\nonumber\\[-9pt]\\[-9pt]
&&{}-\frac{1}{2}\bigl[M^{f_{\ell}},M^{\ell}\bigr]+
\frac{1}{2}f_{\ell
}(t-) \bigl(v_{\ell}(X_t)-v_{\ell}(t-)
\bigr).\nonumber
\end{eqnarray}
Besides, if $u_{\ell}=R_1h$ for some $h$ in $L^2(E,m)$, the left-hand
side of (\ref{aproxsi}) converges in $\mathbf{P}_{g\cdot m}$-measure to
$\int_0^t f_{\ell}(X_s)\,\mathrm{d}C^{\ell}_s$. Therefore, for the general case,
by approximating $u_{\ell}$ for a suite $(R_1h_n)$ with respect to
$\tilde{\mathcal{E}}_1$, it follows thanks to Lemma \ref{l5.1.2FOT}
and Lemma
\ref{lema3.11} that $I^{\ell}_T(t)$ coincides with $\int_0^t f_{\ell
}(X_s)\,\mathrm{d}C^{\ell}_s$ $\mathbf{P}_{g\cdot m}$-a.e. on $\{T<\zeta\}$.

In order to prove the theorem with need to show that there exists a
subsequence of $(\Pi_n)$ such that $\mathbf{P}_{g\cdot m}$-a.e. we have:
For any
$\ell\in\mathbb{N}$, the first five terms in the right-hand side of
(\ref
{aproxsi}) converge to the corresponding terms of the right-hand side
of (\ref{expsi}) uniformly on any compact of $[0,\zeta)$ and the last
term in the right-hand side of (\ref{aproxsi}) converges to zero
uniformly on any compact of $[0,\zeta)$.

We must show only the existence of such subsequence of $(\Pi_n)$ for
the second term in the right-hand side of (\ref{aproxsi}). The
existence of such subsequence for the other terms can be shown using
standard results in the semimartingale theory and the arguments used to
show Lemma \ref{lemavqz}. (See, e.g., Chapter II in \cite{P}.)

For any $n,\ell\in\mathbb{N}$ and $\eta,T>0$ set:
\begin{eqnarray*}
\alpha(n,\ell,T,\eta)&:=&\mathbf{P}_m \biggl(\sup_{t\leq T}
\biggl\llvert \int_t^T\bigl(f_{\ell}
\bigl(Z_T^n(s)\bigr)-f_{\ell}(X_s)
\bigr)\,\mathrm{d}\hat{M}^{\ell
}_s\circ r_T \biggr
\rrvert >\eta;T<\zeta \biggr)
\nonumber
\\
&=&\hat{\mathbf{P}}_m \biggl(\sup_{t\leq T}\biggl\llvert \int
_t^T\bigl(f_{\ell
}\bigl(Z_T^n(s)
\bigr)-f_{\ell}(X_s)\bigr)\,\mathrm{d}\hat{M}^{\ell}_s
\biggr\rrvert >\eta ;T<\zeta \biggr).
\nonumber
\end{eqnarray*}
Using the Doob inequalities, we have:
\begin{eqnarray*}
\alpha(n,\ell,T,\eta)&\leq&\hat{\mathbf{P}}_m \biggl(\biggl\llvert
\int_0^T\bigl(f_{\ell}
\bigl(Z_T^n(s)\bigr)-f_{\ell}(X_s)
\bigr)\,\mathrm{d}\hat{M}^{\ell}_s \biggr\rrvert >
\frac{\eta}{2};T<\zeta \biggr)
\nonumber
\\
&&{}+\hat{\mathbf{P}}_m \biggl(\sup_{t\leq T}\biggl\llvert \int
_0^t\bigl(f_{\ell
}\bigl(Z_T^n(s)
\bigr)-f_{\ell}(X_s)\bigr)\,\mathrm{d}\hat{M}^{\ell}_s
\biggr\rrvert >\frac
{\eta
}{2};T<\zeta \biggr)
\nonumber
\\
&\leq&\frac{4}{\eta}\hat{\mathbf{E}}_m \biggl(\int
_0^T\bigl(f_{\ell
}\bigl(Z_T^n(s)
\bigr)-f_{\ell}(X_s)\bigr)^2\,\mathrm{d}\bigl
\langle\hat{M}^{\ell}\bigr\rangle_s;T<\zeta \biggr)
\nonumber
\\
&\leq&\frac{4}{\eta}\mathbf{E}_m \biggl(\int
_0^T\bigl(f_{\ell
}\bigl(Z_T^n(s)
\bigr)-f_{\ell}(X_s)\bigr)^2\,\mathrm{d}\bigl
\langle\hat{M}^{\ell}\bigr\rangle_s\circ r_T;T<
\zeta \biggr).
\nonumber
\end{eqnarray*}
In view of Remark \ref{nota29}, $\mathbf{P}_m$-a.e. on $\{T<\zeta\}$:
\begin{eqnarray*}
\int_0^T\bigl(f_{\ell}
\bigl(Z_T^n(s)\bigr)-f_{\ell}(X_s)
\bigr)^2\,\mathrm{d}\bigl\langle\hat {M}^{\ell}\bigr
\rangle_s\circ r_T&=&-\int_0^T
\bigl(f_{\ell
}\bigl(Y^n_{T-s}\bigr)-f_{\ell}(X_{T-s})
\bigr)^2\,\mathrm{d}\bigl\langle\hat{M}^{\ell
}\bigr
\rangle_{T-s}
\nonumber
\\
&=&\int_0^T\bigl(f_{\ell}
\bigl(Y^n_{s}\bigr)-f_{\ell}(X_{s})
\bigr)^2\,\mathrm{d}\bigl\langle \hat {M}^{\ell}\bigr
\rangle_{s}
\nonumber
\\
&\leq&\mathrm{e}^T \int_0^{\infty}\mathrm{e}^{-s}
\bigl(f_{\ell}\bigl(Y^n_{s}\bigr)-f_{\ell
}(X_{s})
\bigr)^2\,\mathrm{d}\bigl\langle\hat{M}^{\ell}\bigr
\rangle_{s}.
\nonumber
\end{eqnarray*}
Let $\mu$ be the Revuz measure of $\langle\hat{M}^{\ell} \rangle$,
it follows from Lemma \ref{minuit} that $\mu(E)=2\hat{e}(\hat
{M}^{\ell})<\infty$ where $\hat{e}$ denote the energy with respect
to $(\hat{\mathbf{P}}_x,x\in E)$. Therefore, we have:
\begin{eqnarray*}
\mathbf{E}_m \biggl[ \int_0^{\infty}\mathrm{e}^{-s}
\,\mathrm{d}\bigl\langle\hat {M}^{\ell
}\bigr\rangle_{s}
\biggr]&=&\lim_{x\rightarrow\infty}\mathbf{E}_m \biggl[ \mathrm{e}^{-x}
\bigl\langle\hat{M}^{\ell}\bigr\rangle_{x}+\int
_0^x\mathrm{e}^{-s} \bigl\langle
\hat{M}^{\ell}\bigr\rangle_{s}\,\mathrm{d}s \biggr]
\nonumber
\\
&\leq&\lim_{x\rightarrow\infty}\mathrm{e}^{-x}x \mu(E)+\lim_{x\rightarrow
\infty}\int
_0^x\mathrm{e}^{-s}s\,\mathrm{d}s\mu(E)
\nonumber
\\
&=&\mu(E)<\infty.
\nonumber
\end{eqnarray*}
Since $f_{\ell}$ is quasi-continuous in the strict sense, $f_{\ell
}(Y_t^n)$ converges to $f_{\ell}(X_t)$ uniformly on $\mathbb{R}_+$,
$\mathbf{P}_m$-a.e. Therefore by dominated convergence, we have:
\[
\beta(n,\ell):=\mathbf{E}_m \biggl[\int_0^{\infty}\mathrm{e}^{-s}
\bigl(f_{\ell
}\bigl(Y^n_{s}\bigr)-f_{\ell}(X_{s})
\bigr)^2\,\mathrm{d}\bigl\langle\hat{M}^{\ell
}\bigr
\rangle_{s} \biggr]\rightarrow0\qquad \mbox{as }n\rightarrow\infty.
\]
For any $j\in\mathbb{N}$ let $n_j$ such that:
\[
\sum_{\ell=1}^{\infty}\frac{1}{\ell^2}\beta
\bigl(n_j,j^{-1}\bigr)\leq\frac
{1}{j^3}\qquad  \forall
n\in\mathbb{N}.
\]
Since for all $\eta,T>0$: $\alpha(n,l,T,\eta)\leq\frac{4}{\eta
}\mathrm{e}^T\beta(n,\ell)$, we have that $\alpha(n_j,\ell,T,j^{-1})\leq
\frac{4}{j^2}\mathrm{e}^T$ $\forall T>0$. It follows from Borel--Cantelli lemma
that for any $T,\ell$: $\mathbf{P}_m(\Omega\setminus\Omega_{T,\ell
})=0$ where:
\begin{eqnarray}
\Omega_{T,\ell}&:=&\{\zeta\leq T\}
\nonumber
\\
&&{} \cup\biggl\{ \int_0^t\bigl(f_{\ell}
\bigl(Z_T^{n_j}(s)\bigr)-f_{\ell}(X_s)
\bigr)\,\mathrm {d}\hat {M}^{\ell}_s\rightarrow0\mbox{
uniformly on any compact of }[0,T];T<\zeta \biggr\},
\nonumber
\end{eqnarray}
then $\mathbf{P}_m(\Omega\setminus\Omega^*)=0$ where $\Omega^*=\bigcap_{T\in Q_+,\ell\in\mathbb{N}}\Omega_{T,\ell}$. It is easy to show
that for
$\omega\in\Omega^*$, $\int_0^t(f_{\ell}(Z_T^{n_j}(s))-f_{\ell
}(X_s))\,\mathrm{d}\hat{M}^{\ell}_s$ converges to zero uniformly on any
compact of $[0,\zeta(\omega))$.
\end{pf*}

\begin{ejem}
In this example, we show that the stochastic integral constructed
by Chen \textit{et al.} \cite{CFKZ} for symmetric Dirichlet forms can be
defined in the sense of Definition \ref{intsto}. Moreover, both
definitions coincide $\mathbf{P}_m$-a.e. $\llbracket0,\zeta
\llbracket$. We use the
notations and definitions of \cite{CFKZ}, thus $\Lambda$ is a linear
operator that maps some class of local MAF's into even local CAF's
admitting $m$-null set. Let $M$ be a locally square-integrable MAF on
$\llbracket0,\zeta\llbracket$ that belongs to the domain of $\Lambda
$. We see
from the proof of \cite{CFKZ}, Theorem 3.7 and Lemma 3.2, that there
exists a nest $\{F_k\}$ of closed sets such that $\mathbf{P}_m$-a.e.
on $\llbracket
0,\tau_{F_k}\llbracket$:
%
\begin{equation}\label{paraeje}
\Lambda(M)=\Lambda\bigl(M^k\bigr)+A_t^k+L_t^k,
\end{equation}
where $M^k\in\M$, $A$ is a CAF of bounded variation and $L^k\in
(\mathcal{M}_{\mathit{loc}})^{\llbracket0,\zeta\llbracket}$. With a
refinement argument
used in the proof of \cite{CFKZ}, Lemma 4.6, one checks that
$\Lambda(M)$ is a local CAF of $X$. Denote by $\mathfrak{E}$ the set
of CAF of $X$ of finite energy. In view of \cite{CFKZ}, Proposition
2.8, the right-hand side of (\ref{paraeje}) belongs to $\mathfrak
{E}_{f\mbox{-}\mathit{loc}}$, hence $\Lambda(M)$ belongs to $(\mathfrak
{E}_{f\mbox{-}\mathit{loc}})_{f\mbox{-}\mathit{loc}}=\mathfrak{E}_{f\mbox{-}\mathit{loc}}$.

By \cite{CFKZ}, Theorem 3.7, $\Lambda(M)$ is of zero quadratic
variation in the sense of Definition \ref{defafzqv}. Then $\Lambda(M)$
belongs to $\mathcal{N}_{f\mbox{-}\mathit{loc} }$ and therefore the integral
$f*\Lambda(M)$ is well defined for any $f\in\mathcal{F}_{\mathit{loc}}$.

The stochastic integral defined in \cite{CFKZ} can be approximate in
some sense by Riemann sums. (See \cite{CFKZ}, Theorem 4.4.)
Consequently, thanks to Theorem \ref{integralestocastica} the
integrals $f*\Lambda(M)$ given by \cite{CFKZ} and Definition \ref
{intsto} both coincide $\mathbf{P}_m$-a.e. on $\llbracket0,\zeta
\llbracket$ for any
$f\in\mathcal{F}_b$ and therefore for any
$f\in\mathcal{F}_{\mathit{loc}}$.
\end{ejem}

\section{\texorpdfstring{Proof of Propositions \protect\ref{itoLevy} and \protect\ref{ito}}
{Proof of Propositions 1.2 and 1.3}}\label{prueba2}

\begin{pf*}{Proof of Proposition \ref{itoLevy}}
From the proof of Lemma 1.1 in \cite{W2}, there exist sequences
$(u_n)$, $(g_n)$ in $\mathcal{F}$ and nest of nearly Borel finely open sets
$(G_n)$ and $(\mathcal{G}_n)$ such that for any $n$: $u(x)=u_n(x)$
q.e. on
$G_n$, $G_n\subset\mathcal{G}_n$, $g_n(x)=1$ q.e. on $G_n$,
$g_n(x)=0$ q.e. on
$E\setminus\mathcal{G}_n$ and $\|g_n\|_{\infty}\leq1$. Moreover, there
exists a sequence of positive numbers $(\varepsilon_n)$ converging to
$0$ that the following limit define and element in $\M_{f\mbox
{-}\mathit{loc}}$, where the convergence is uniformly on any compact of
$[0,\zeta)$ $\mathbf{P}_x$-a.e. for q.e. $x\in E$.
\begin{eqnarray*}
M_t^{bj}&:=&\lim_{n\rightarrow\infty}\sum
_{s\leq
t}\bigl[u(X_s)-u(X_{s-})
\bigr]1_{\{\varepsilon_n <|\Delta u(X_s)|< 1\}}
\nonumber
\\
&&{}-\int_0^t\int_E
1_{\{\varepsilon_n<|u(y)-u(X_s)|<1\}
}\bigl[u(y)-u(X_{s})\bigr]N(X_s,
\mathrm{d}y)\,\mathrm{d}H_s, \qquad  t<\zeta .
\nonumber
\end{eqnarray*}
Besides, it is known that $M_t^c:= M^{u_n,c}_t \mbox{ if }t<\tau_{\mathcal{G}
_n}$ define a local CAF in $\M_{f\mbox{-}\mathit{loc}}$. Then set
$W^u=M^{bj}-M^c$. We shall proof that $C^u\in\mathcal{N}_{c,f\mbox
{-}\mathit{loc}}$, where:
\[
C^u_t:=u(X_t)-u(X_0)-\sum
_{s\leq t}\bigl[u(X_s)-u(X_{s-})
\bigr]1_{\{|\Delta
u(X_s)|\geq1\}}-W^u, \qquad t<\zeta.
\nonumber
\]
For any $\ell\in\mathbb{N}$, by taking a subsequence of
$(\varepsilon_n)$,
if necessary, we have that $\mathbf{P}_x$-a.e. for q.e. $x\in E$. For all
$t<\tau_{G_{\ell}}$,
\begin{eqnarray*}
&&C^u_t-N_t^{u_\ell}-\int
_0^t u_{\ell}(X_s)N(X_s,
\partial)\,\mathrm{d} H_s
\nonumber
\\
&&\qquad {}+\int_0^t\int_{E}1_{\{|u_{\ell}(X_s)-u_{\ell}(y)|\geq1\}
}
\bigl[u_{\ell}(y)-u_{\ell}(X_s)\bigr]N(X_s,
\mathrm{d}y)\,\mathrm {d}H_s
\nonumber
\\
&&\quad =\lim_{n\rightarrow\infty} \biggl( \int_0^t\int
_E 1_{\{
\varepsilon_n<|u(y)-u(X_s)|<1\}}\bigl[u(y)-u(X_{s})
\bigr]N(X_s,\mathrm {d}y)\,\mathrm{d}H_s
\nonumber
\\
&&\hspace*{46pt}{} - \int_0^t\int
_E 1_{\{\varepsilon_n<|u_{\ell
}(y)-u_{\ell}(X_s)|<1\}}\bigl[u_{\ell}(y)-u_{\ell}(X_{s})
\bigr]N(X_s,\mathrm {d}y)\,\mathrm{d} H_s \biggr)
\nonumber
\\
&&\quad =\lim_{n\rightarrow\infty} \biggl( \int_0^t\int
_Eg_{\ell}(X_s) 1_{\{\varepsilon_n<|u(y)-u(X_s)|<1\}}
\bigl[u(y)-u(X_{s})\bigr]N(X_s,\mathrm {d}y)\,\mathrm{d}
H_s
\nonumber
\\
&&\hspace*{46pt}{} - \int_0^t\int
_E g_{\ell}(X_s)1_{\{\varepsilon
_n<|u_{\ell}(y)-u_{\ell}(X_s)|<1\}}
\bigl[u_{\ell}(y)-u_{\ell
}(X_{s})\bigr]N(X_s,
\mathrm{d}y)\,\mathrm{d}H_s \biggr)
\nonumber
\\
&&\quad =-\int_0^t\int_E1_{\{|u(y)-u(X_s)|<1\}}
\bigl[g_{\ell}(y)-g_{\ell
}(X_s)\bigr]
\bigl[u(y)-u(X_{s})\bigr]N(X_s,\mathrm{d}y)
\,\mathrm{d}H_s
\nonumber
\\
&&\qquad {} +\int_0^t\int
_E 1_{\{|u_{\ell}(y)-u_{\ell}(X_s)|<1\}
}\bigl[g_{\ell}(y)-g_{\ell}(X_s)
\bigr] \bigl[u_{\ell}(y)-u_{\ell
}(X_{s})
\bigr]N(X_s,\mathrm{d} y)\,\mathrm{d}H_s
\nonumber
\end{eqnarray*}
and the last term belongs to $\mathcal{N}_{f\mbox{-}\mathit{loc}}$, in fact, for
$t<\tau_{G_{\ell}}$:
\begin{eqnarray*}
&&\int_0^t\int_E1_{\{|u(y)-u(X_s)|<1\}}\big|g_{\ell}(y)-g_{\ell
}(X_s)\big|\big|u(y)-u(X_{s})\big|N(X_s,
\mathrm{d}y)\,\mathrm{d}H_s
\nonumber
\\
&&\quad =\int_0^t\int_E1_{\{|u(y)-u(X_s)|<1\}}g_{\ell}(X_s)\big|g_{\ell
}(y)-g_{\ell}(X_s)\big|
\big|u(y)-u(X_{s})\big|N(X_s,\mathrm{d}y)\,\mathrm
{d}H_s
\nonumber
\\
&&\quad \leq \int_0^t\int_E
\bigl[g_{\ell}(X_s)-g_{\ell}(y)\bigr]^2
N(X_s,\mathrm{d} y)\,\mathrm{d}H_s
\nonumber
\\
&&\qquad {}+\int_0^t\int_Eg_{\ell}(y)\big|g_{\ell}(y)-g_{\ell}(X_s)\big|\big|u_{\ell
}(y)-u_{\ell}(X_{s})\big|N(X_s,
\mathrm{d}y)\,\mathrm{d}H_s
\nonumber
\\
&&\quad <\infty.
\nonumber
\end{eqnarray*}
Therefore, $C^u\in(\mathcal{N}_{f\mbox{-}\mathit{loc}})_{f\mbox
{-}\mathit{loc}}=\mathcal{N}_{f\mbox{-}\mathit{loc}}$.
\end{pf*}

\begin{pf*}{Proof of Proposition \ref{ito}}
Thanks to Theorem \ref{integralestocastica}, the It\^{o} formula can
be proved up to $\zeta$ with the same argument used to prove the
generalized It\^{o} formula of \cite{CFKZ}. (Theorem 4.7. of \cite
{CFKZ}.) When $u\in\mathcal{F}$, all terms in the decomposition of Theorem
\ref{itoLevy} are finite AF. Moreover, $C^u_t=N^u_t-\int_0^t\int_{E}1_{\{|u(X_s)-u(y)|\geq1\}}[u(y)-u(X_s)]N(X_s,\mathrm{d}y)\,\mathrm
{d}H_s$, then
it follows from Remark \ref{nota101}(iii) that the stochastic
integrals $\int_0^t \frac{\partial\Phi}{\partial
x_i}(u(X_s))\,\mathrm{d}
C_s^u$ are finite AF. Besides $W^u$ belongs to $\M$ and therefore the
integrals $\int_0^t \frac{\partial\Phi}{\partial
x_i}(u(X_s))\,\mathrm{d}
W^u_s$ are also finite AF. Therefore, when $u\in\mathcal{F}$ all
terms present
in the It\^{o} formula are finite AF, then the It\^{o} formula can be
extended from $[0,\zeta)$ to $[0,\infty)$.
\end{pf*}

\section*{Acknowledgements}
This work was part of my Ph.D. thesis realized in the university of
Paris VI. I would therefore like to thank sincerely my Ph.D. advisor
Nathalie Eisenbaum for every helpful discussion that led to improvement
of the results in this paper. In addition, I would also like to thank
the Associated Editor for his/her suggestions.


%

\printhistory

\end{document}